\documentclass[12pt]{article}

\usepackage{amssymb}
\usepackage{epsfig}
\usepackage{amsmath}
\usepackage{amsthm}
\usepackage{subfigure}
\usepackage{url}
\usepackage{color}
\usepackage{graphicx}
\usepackage{fancyhdr}
\usepackage{tikz}
\usepackage{float}
\usepackage{cite}
\usepackage[section]{placeins}
\usepackage{wrapfig}
\usetikzlibrary{calc}

\DeclareMathOperator{\transverse}{\cap\kern-7.75pt\top}
\DeclareMathOperator{\area}{Area}
\DeclareMathOperator{\argmax}{arg max}

\newcommand{\D}{\operatorname{D}}
\newcommand{\tc}[2]{T_{#1}(#2)}
\newtheorem{thm}{Theorem}
\newtheorem{cor}[thm]{Corollary}

\newtheorem{lem}[thm]{Lemma}
\newtheorem{rem}[thm]{Remark}
\newtheorem{clm}{Claim}
\newtheorem{dfn}[thm]{Definition}

\newcommand{\pd}[2]{\frac{\partial #1}{\partial #2}}
\newcommand{\disk}[2]{D(#1, #2)}
\newcommand{\cir}[2]{C(#1, #2)}
\newcommand{\shape}{\Omega}
\newcommand{\bd}{\partial \shape}
\newcommand{\tgl}{tangentially graph-like}
\newcommand{\tcgl}{tangent-cone graph-like}

\newcommand{\gm}[1]{\gamma (#1)}
\newcommand{\dgm}[1]{{\gamma'}(#1)}
\newcommand{\dgmm}{{\gamma'}}
\newcommand{\delep}{\delta_{\epsilon}}
\newcommand{\enu}{\epsilon_\nu}
\newcommand{\rhh}{\hat{\rho}}
\newcommand{\prff}{\noindent {\bf Proof: }\setcounter{clm}{0}}
\newcommand{\prfclm}[1]{\noindent {\bf Proof of Claim \ref{#1}: }}
\newcommand{\epf}{$\blacksquare \;\;$}
\newcommand{\epfclm}{$\Box \;\;$}

\title{Nonasymptotic densities for shape reconstruction}
\author{Sharif Ibrahim\footnote{Department of Mathematics, Washington State University, Pullman, WA 99164-3113, USA} \and Kevin Sonnanburg\footnote{Department of Mathematics, University of Tennessee Knoxville, Knoxville, TN 37996-1320}
\and Thomas J. Asaki\footnotemark[1]
\and Kevin R. Vixie\footnotemark[1]
}

\usepackage[bookmarks=true, %
            bookmarksnumbered=false, %
            bookmarksopen=false, %
            colorlinks=true, 
            linkcolor=red]{hyperref}

\begin{document}

\maketitle
\begin{abstract}
In this work, we study the problem of reconstructing shapes from 
simple nonasymptotic densities measured only along shape boundaries.
The particular density we study is also known as the integral area invariant and corresponds to the 
area of a disk centered on the boundary that is also inside the shape.
It is easy to show uniqueness when these densities 
are known for all radii in a neighborhood of $r = 0$, but much 
less straightforward when we assume that we only know the area invariant and its derivatives for only one $r > 0$.
We present variations of uniqueness results for reconstruction (modulo translation and rotation) of polygons and (a dense set of) smooth curves under certain regularity conditions.
\end{abstract}

\section{Introduction}
\label{sec:intro}

This work discusses the integral area invariant introduced by Manay et al.\cite{manay-2004-1}, particularly with regard to reconstructability of shapes.
This topic has been considered previously by Fidler et al.\cite{fidler-2008-1}\cite{fidler-2012-1} for the case of star-shaped regions.
Recent results have shown local injectivity in the neighborhood of a circle \cite{bauer-2012-1} and for graphs in a neighborhood of constant functions \cite{calder-2012-1}.

The present work does not assume a star-shaped condition but does make use of a \tcgl{} condition which is local to the integral area circle.
We also present an interpretation of the integral area invariant as a nonasymptotic density.
This is based on a poster presented by the authors\cite{ibrahim-poster-2010}.

Our \tgl{} and \tcgl{} conditions (definitions \ref{def:tgl} and \ref{def:tcgl} in section \ref{sec:notation}) restrict our attention to shapes with boundaries that can locally (i.e., within radius $r$) be viewed as graphs of functions in a Cartesian plane in one particular orientation (in the case of \tgl{}) or a particular set of orientations (for \tcgl{}).
Intuitively, these conditions guarantee that the boundary does not turn too sharply within the given radius and that working locally in Euclidean space is the same as working locally on the boundary of our shapes (i.e., the shape boundary does not pass through any given invariant circle multiple times, section \ref{sec:tgl_2arc}).
These simplifying assumptions allow us to explicitly analyze what happens when we move along the boundary and to work locally without worrying about global effects.

We show that the \tcgl{} property can be preserved when approximating a shape with a polygon (section \ref{sec:tgl-ncgl}) and discuss what the derivatives of these nonasymptotic densities represent (section \ref{sec:derivatives}) and show that all \tgl{} boundaries can be reconstructed (modulo translations and rotations) given sufficient information about the nonasymptotic density and its derivatives (section \ref{sec:reconstruct-T} and appendix \ref{sec:appen}).

The main contribution of this paper is to show (under our \tcgl{} condition) that all polygons (theorem \ref{thm:tcgl-poly-reconstruct} in section \ref{sec:polygon-no-tail}) and a $C^1$-dense set of $C^2$ boundaries (theorem \ref{thm:tcgl-reconstruct} in section \ref{sec:generic}) are reconstructible (modulo translations and rotations).
We briefly discuss and sketch the proofs of these two theorems.

{
\renewcommand{\thethm}{\ref{thm:tcgl-poly-reconstruct}}
\begin{thm}
For a polygon $\shape$ which is \tcgl{} with radius $r$, suppose that we have the integral area invariant $g(s,r)$ where $s$ is parameterized by arc length.
Suppose that for all $s$ we know $g(s,r)$ and its first derivatives with respect to $r$ (disk radius) and $s$ (position along the boundary).
This information is sufficient to completely determine $\shape$ up to translation and rotation; that is, we can recover the side lengths and angles of $\shape$.
\end{thm}
\addtocounter{thm}{-1}
}

The proof of this theorem uses the discontinuities in the $s$ derivative to determine the locations of vertices (and thus the side lengths between them).
We combine the $r$ derivative and the one-sided $s$ derivative information when centered on a vertex to recover the angles at which the polygon enters and exits the circle (which might not be the polygon vertex angle if the circle contains another vertex).
Doing this with the other one-sided $s$ derivative gives the same thing but using the orientation determined by the other polygon side incident to the vertex.
The combination of these yields the polygon's angle at each vertex.

{
\renewcommand{\thethm}{\ref{thm:tcgl-reconstruct}}
\begin{thm}
  Define $\Bbb{G} \equiv \{ \gamma| \gamma$ is a $C^2$ simple closed
  curve and \tgl{} for $r=\hat{r}\}$.  Suppose that, for $r=\hat{r}$, for
  all $s\in[0,L]$, and for each $\gamma\in\Bbb{G}$, we know the first-,
  second-, and third-order partial derivatives of $g_{\gamma}(s,r)$.
  Then the set of reconstructible $\gamma\in\Bbb{G}$ is
  $C^1$ dense in $\Bbb{G}$ where reconstructability is modulo reparametrization, translation, and rotation.
\end{thm}
\addtocounter{thm}{-1}
}

The first part of the proof shows that the derivative information can be used to obtain the curvature.
However, it is not the curvature at the boundary point where the circle is centered but rather the curvature at each of the points where the boundary enters and exits the circle.
Although the Euclidean distance to these points is known, the arc length distances are not and can vary from point to point.
Thus the sequences of curvatures we obtain also lose the arc length parameterization of our area invariant.
The rest of the proof is concerned with finding the arc length distance from the center to the entry and exit points which effectively recovers the curvature for all points.
This relies on matching up the unique features of exit angle sequences with each other which in turn relies on the existence of unique maxima and minima in these sequences.
While this is not true in general, it can be arranged to be so by a suitable small perturbation of the boundary (which is why our result is one of density rather than for all shapes).

This is a theoretical paper about a measure that is useful in applications: we do
not pretend that the reconstruction techniques in our proofs are practically
useful. In fact, the reconstructions we use to show uniqueness would be
seriously disturbed by the noise that any practical application would
encounter.
We do, however, comment on some possible approaches to
reconstruction (section \ref{sec:numerics}) using the \textsc{OrthoMads} direct search algorithm\cite{AbAuDeLe2009} to successfully reconstruct shapes which are not predicted by our theory.

\section{Notation and Preliminaries}
\label{sec:notation}

\begin{figure}[!]
\centering
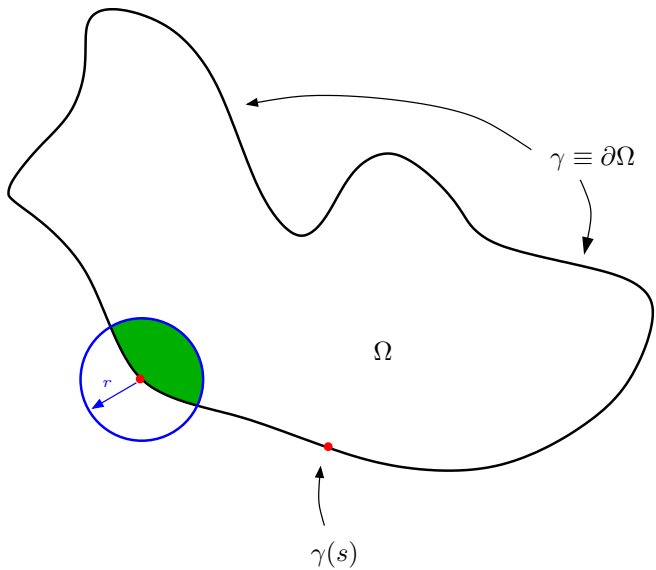
\caption{Notation and basic setup}
\label{fig:shape-notation}
\end{figure}

Unless otherwise specified, we will be assuming throughout this paper that $\shape \subset \mathbb{R}^2$ is a compact set with simple closed, piecewise continuously differentiable boundary $\bd$ of length $L$.
Let $\gamma : [0, L] \rightarrow \bd$ be a continuous arclength parameterization of $\bd$ (see Figure~\ref{fig:shape-notation}).
We will adopt the convention that $\gamma$ traverses $\bd$ in a counterclockwise direction so it always keeps the interior of $\shape$ on the left (there is no compelling reason for this particular choice, but adopting a consistent convention allows us to avoid some ambiguities later).
Note that $\gamma(0) = \gamma(L)$ and that $\gamma$ restricted to $[0,L)$ is a bijection.
Denote by $\disk{p}{r}$ the closed disk and $\cir{p}{r}$ the circle of radius $r$ centered at the point $p \in \mathbb{R}^2$.

In geometric measure theory, the $m$-dimensional density of a set $A \subseteq \mathbb{R}^n$ at a point $p \in \mathbb{R}^n$ is given by
\[
\Theta^m(A, p) = \lim_{r\downarrow 0}\frac{\mathcal{H}^m(A \cap \disk{p}{r})}{\alpha_m r^m}
\]
where $\mathcal{H}^m$ is the $m$-dimensional Hausdorff measure
and $\alpha_m$ is the volume of the unit ball in $\mathbb{R}^m$\cite{morgan-gmt}.
In the current context, the $2$-dimensional density of $\shape$ at $\gamma(s)$ is simply
\[
\Theta^2(\shape, \gamma(s)) = \lim_{r\downarrow 0}\frac{\area(\shape \cap \disk{\gamma(s)}{r})}{\pi r^2}.
\]

While we can evaluate this for all $s \in [0,L)$, just knowing the density at every point along the boundary is generally insufficient to reconstruct the original shape.
If $\gamma'(s)$ exists, then $\area(\shape \cap \disk{\gamma(s)}{r})$ is approximated arbitrarily well for sufficiently small $r$ by replacing $\bd$ with its tangent line (which gives us an area of exactly $\frac{\pi r^2}{2}$).
Hence, we have $\Theta^2(\shape, \gamma(s)) = \frac{1}{2}$ at any point where $\gamma$ is differentiable.
That is, just knowing $\Theta^2$ (i.e., the limit) is insufficient to distinguish any two shapes with $C^1$ boundary.

Contrast this with the situation where we know $\area(\shape \cap \disk{\gamma(s)}{r})$ for every $s \in [0,L)$ and $r > 0$ (i.e., we have all of the values needed to compute the limit as well).
This added information is sufficient to uniquely identify $C^2$ curves by recovering their curvature at every point (see Appendix \ref{sec:appen}).

One natural question to ask (and the focus of the present work) is whether failing to pass to the limit (i.e., using some fixed radius $r$ instead of the limit or all $r > 0$) and collecting the values for all points along the boundary preserves enough information to reconstruct the original shape.
That is, can a nonasymptotic density (perhaps along with information about its derivatives) be used as a signature for shapes?

\subsection{Definitions}

\begin{dfn}
\label{def:geomeasure}
In the current context, the integral area invariant\cite{manay-2004-1} is denoted by $g : [0, L) \times \mathbb{R}^+ \rightarrow \mathbb{R}^+$ and given by
\[
g(s, r) = \int_{\disk{\gamma(s)}{r}\cap\shape} \,dx = \area(\shape \cap \disk{\gamma(s)}{r}).
\]
\end{dfn}

\begin{rem}
Note the lack of the normalizing factor $\pi r^2$ in the definition of $g(s,r)$.
Since we presume that $r$ is fixed and known for the situations we study, it's trivial to convert data between the forms $g(s,r)$ and $\frac{g(s,r)}{\pi r^2}$;
we choose to leave out the normalizing factor in the definition of $g(s,r)$ as it is the integral area invariant of Manay et al.\cite{manay-2004-1} and this form proves useful when computing derivatives in section \ref{sec:derivatives}.
\end{rem}

We introduce the \tgl{} condition as a simplifying assumption for the shapes we consider.

\begin{dfn}
\label{def:gl}
\label{def:tgl}
For a fixed radius $r$, we say that $\bd$ is graph-like (GL) at a point $p \in \bd$ (or graph-like on $\disk{p}{r}$) if it is possible to impose a Cartesian coordinate system such that the set of points $\bd \cap \disk{p}{r}$ is the graph of some function $f$ in this coordinate system.
Without loss of generality, we adopt the convention that $p$ is the origin so that $f(0) = 0$.
We define \tgl{} (TGL) in the same way but further require that $\bd$ be continuously differentiable and $f'(0) = 0$ (noting that $f$ is $C^1$ because $\bd$ is).
This is illustrated in figure \ref{fig:tgl}.
Without loss of generality (and in keeping with our convention that $\gamma$ traverses $\bd$ counterclockwise), we assume that the interior of $\shape$ is ``up'' in the circle (i.e., that $(0,\epsilon) \in \shape$ for sufficiently small $\epsilon > 0$).
If $\bd$ is (tangentially) graph-like on $\disk{p}{r}$ for all $p \in \bd$, we say that $\bd$ is (tangentially) graph-like for radius $r$.
\end{dfn}
\begin{figure}
\makebox[\textwidth][c]{
\subfigure[]{
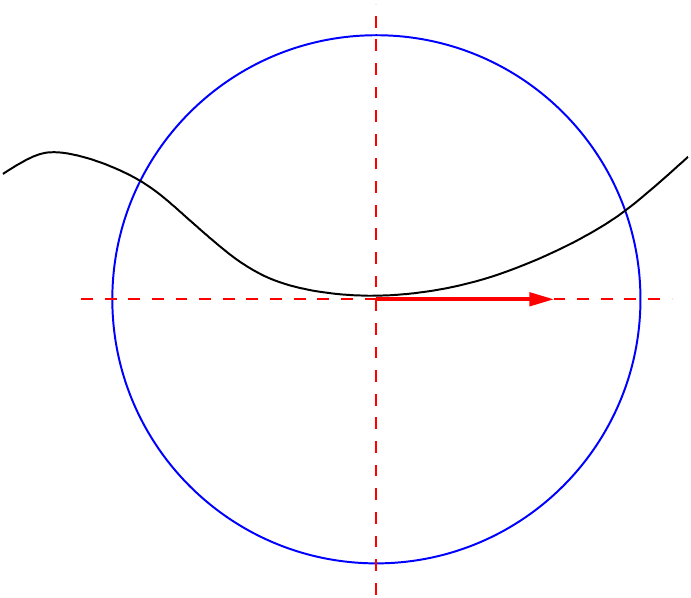
\label{fig:tgl}
}
\subfigure[]{
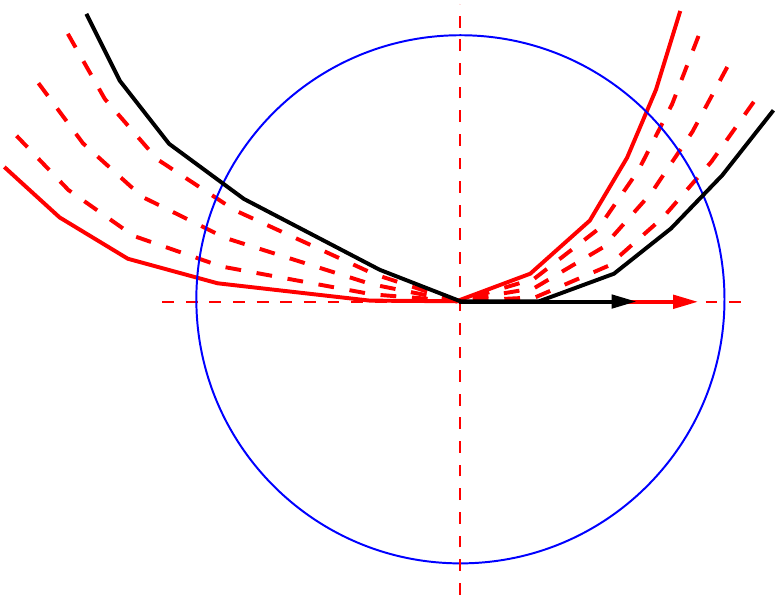
\label{fig:tcgl}
}}
\caption{\subref{fig:tgl} Tangentially and \subref{fig:tcgl} tangent cone graph-like}
\end{figure}
\begin{figure}
\subfigure[]{
\label{fig:glfail}
\begin{tikzpicture}[scale=.9]
\draw[thick=2pt,fill=teal!40] (-2, 4) -- (-2,0) -- (2,0) -- (2,4) -- cycle;
\draw[color=blue] (0,0) circle (3);
\node[label=below:{$\gamma(s)$},style=circle,fill=red,minimum size=4pt,inner sep=0pt] () at (0,0) {};
\end{tikzpicture}}
\subfigure[]{
\label{fig:tglfail}
\begin{tikzpicture}[scale=.9]
\draw [thick=2pt,rounded corners=5mm,fill=teal!40] (-4,0)--(4,0)--(4,2)--(-4,2)--cycle;
\draw[color=blue] (0,0) circle (3);
\node[label=below:{$\gamma(s)$},style=circle,fill=red,minimum size=4pt,inner sep=0pt] () at (0,0) {};
\end{tikzpicture}}
\caption{\subref{fig:glfail} The square is not graph-like with the indicated radius (no orientation makes it a graph).
\subref{fig:tglfail} The rounded rectangle is graph-like but not \tgl{} with the indicated center and radius.}
\end{figure}
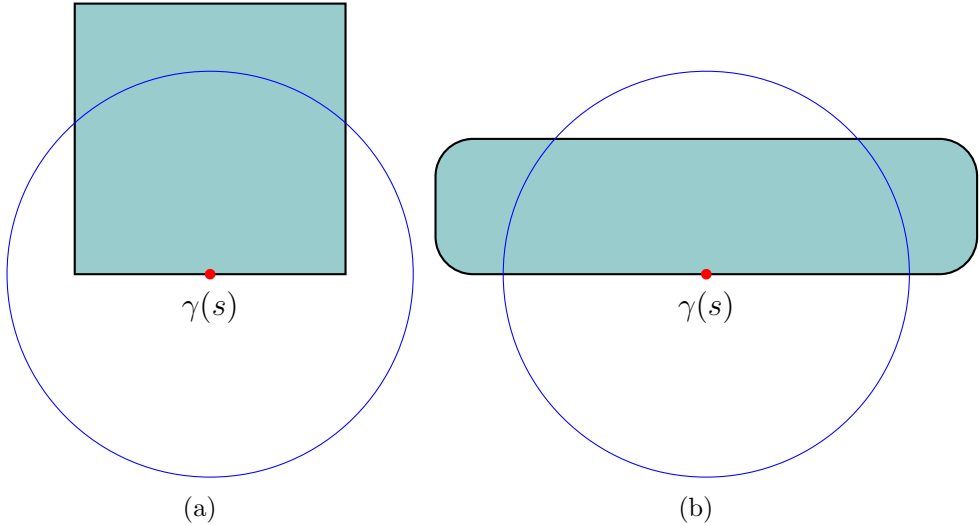

It is instructive to consider what is {\em not} graph-like or \tgl{}.
Violations of the graph-like condition are generally due to a radius that is too large (certainly, choosing a radius so large that all of $\shape$ is in the disk will do it).
For example, a unit side length square is not graph-like with radius $\frac{1}{2} + \epsilon$ for any $\epsilon > 0$ (position the circle at the center of a side; see figure \ref{fig:glfail}).
Notice that the same square is graph-like with any radius $\frac{1}{2}$ or below.
A shape can fail to be \tgl{} while still being graph-like if it fails to be a graph in the required orientation but works in some other (see figure \ref{fig:tglfail}).

We would like to consider shapes with corners but our \tgl{} condition requires that the boundary be differentiable everywhere.
The following definitions allow us to generalize the \tgl{} condition to this situation by using one-sided derivatives.

\begin{dfn}
\label{def:tcone}
Given a piecewise $C^1$ function $\gamma : [0,L] \rightarrow \mathbb{R}^2$, we define the tangent cone of $\gamma$ at a point $s$ (which we denote by $\tc{\gamma}{s}$) in terms of the one-sided derivatives.
In particular, we let $\tc{\gamma}{s} = \{\alpha \Gamma^- + \beta \Gamma^+ \mid \alpha,\beta \geq 0, \alpha + \beta > 0\}$ where $\Gamma^- = \lim_{t\uparrow s} \gamma'(t)$ and $\Gamma^+ = \lim_{t\downarrow s} \gamma'(t)$.
\end{dfn}

\begin{dfn}
\label{def:tcgl}
We extend the \tgl{} notion to boundaries that are piecewise $C^1$ by defining $\bd$ to be \tcgl{} (TCGL) at a point $\gamma(s) \in \bd$ if it is graph-like at $\gamma(s)$ for every orientation in the tangent cone of $\bd$ at $s$.
More precisely, for every 
$w \in \tc{\gamma}{s}$ and every pair of distinct points $u,v\in\bd\cap\disk{p}{r}$, we have $\langle w, u-v \rangle \neq 0$ (see figure \ref{fig:tcgl}).
\end{dfn}

\begin{rem}
It is clear that $\tc{\gamma}{s}$ in definition \ref{def:tcone} is a convex cone.
The tangent cone is dependent on the direction in which $\gamma$ traverses $\bd$ (which by convention was counterclockwise) since an arc-length traversal $\hat{\gamma}(s,r) = \gamma(L-s,r)$ would have different tangent cones (namely, $w \in \tc{\gamma}{s}$ iff $-w \in \tc{\hat{\gamma}}{s}$).
However, these differences are irrelevant to the application of definition \ref{def:tcgl}.
\end{rem}

\begin{rem}
Note that when $\bd$ is $C^1$, there is only one direction in $\tc{\gamma}{s}$ for each $s$ (i.e., the tangent to $\bd$ at $\gamma(s)$).
Thus, the definitions of \tgl{} and \tcgl{} coincide when $\bd$ is $C^1$ and every \tgl{} boundary is \tcgl{}.
\end{rem}

\subsection{Two-Arc Property}
\label{sec:tgl_2arc}

The graph-like condition implies (in proof of the following lemma) that $\shape$ will never be entirely contained in the disk, no matter where on the boundary we center it.
That is, some part of $\shape$ lies outside of $\disk{p}{r}$ for every $p \in \bd$.

\begin{lem}
\label{2intlemma}
Let $r \in \mathbb{R}^+$ and $p \in \bd$. If $\bd$ is graph-like on $\disk{p}{r}$, then 
$|\bd \cap \cir{p}{r}| \geq 2$.
\begin{proof}
Suppose by way of contradiction that $|\bd \cap \cir{p}{r}| < 2$.
Since $\bd$ is a simple closed curve, we have $\bd \subseteq \disk{p}{r}$.
As $\bd$ is graph-like at $p$ with radius $r$, there exists some orientation for which $\bd \cap \disk{p}{r} = \bd$ is the graph of a well-defined function.
However, $\bd$ is a simple closed curve so it is not the graph of a function in any orientation, yielding a contradiction.
\end{proof}
\end{lem}

The next result is the reason we find the \tcgl{} condition useful.
It says that if $\bd$ is \tcgl{} with radius $r$, then, for every $p \in \bd$, the disk $\disk{p}{r}$ has only two points of intersection with $\bd$ and these are transverse.
In other words, this means that when working locally in the disk $\disk{p}{r}$ we need only consider a single piece of $\bd$.

\begin{thm}
\label{thm:two-intersections}
If $\bd$ is \tcgl{} with radius $r \in \mathbb{R}^+$ at $p \in \bd$,
then $|\bd \cap \cir{p}{r}| = 2$ and $\bd$ crosses $\cir{p}{r}$ transversely at these points.
As a result, for every $q_1, q_2 \in \bd \cap \disk{p}{r}$, there is a unique arc along $\bd$ between them in $\disk{p}{r}$.
\begin{proof}
By Lemma \ref{2intlemma}, we have that $|\bd \cap \cir{p}{r}| \geq 2$.
Note that $\bd$ contains an interior point ($p$) and at least two boundary points of the disk $\disk{p}{r}$ (since $|\bd \cap \cir{p}{r}| \geq 2$).
As $\bd$ is connected and simply closed, there must exist an arc of $\bd$ within the disk going from some point on $\cir{p}{r}$ through $p$ to another point on $\cir{p}{r}$.

Suppose $|\bd \cap \cir{p}{r}| > 2$; that is, there are other points of intersection.
Letting $q$ denote one of these, there are two cases to consider (illustrated in Figure \ref{fig:nointersect}).
\begin{figure}
\centering
\subfigure{
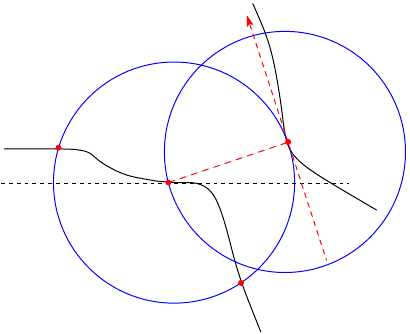
}
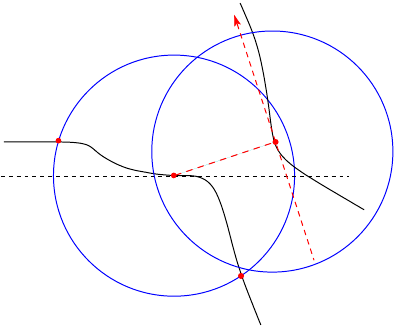
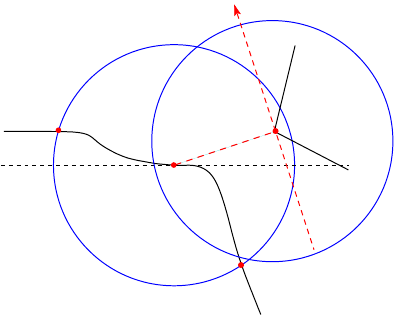
\caption{Additional points of intersection violate the TCGL condition.}
\label{fig:nointersect}
\end{figure}
\begin{enumerate}
\item [(a)]
\label{case:nocross}
$\bd$ does not cross $\cir{p}{r}$ at $q$.

As $\bd$ is \tcgl{} at $q$, then $\bd\cap\cir{q}{r}$ is a graph in every orientation in the tangent cone of $\bd$ at $q$.
In particular, note that the tangent line to $\cir{p}{r}$ at $q$ is in this cone.
However, the line from $p$ to $q$ is normal to this line and thus $\bd\cap\cir{q}{r}$ is not graph-like in this orientation, a contradiction.
Therefore, this case cannot occur.
This argument applies to all points in $\bd \cap \cir{p}{r}$ so we immediately have the result that $\bd$ always crosses $\cir{p}{r}$ transversely.

\item [(b)]
$\bd$ crosses $\cir{p}{r}$ at $q$.

There exists $q' \in \bd \cap \cir{p}{r}$ such that there is a path along $\bd$ in $\disk{p}{r}$ from $q$ to $q'$.
That is, there exist $s_1, s_2 \in [0,L)$ (without loss of generality, $s_1<s_2$) such that $\gamma(s_1) = q$, $\gamma(s_2) = q'$ and the image of $[s_1,s_2]$ under $\gamma$ is contained in $\disk{p}{r}$ (but does not include $p$, since it is on another arc and $\bd$ is simple).
Thus $\gamma$ enters $\cir{p}{r}$ at $s_1$ and exits at $s_2$.

If we can find $s \in [s_1, s_2]$ and $w$ in the tangent cone of $\bd$ at $\gamma(s)$ satisfying $\langle w, p - \gamma(s) \rangle  = 0$, we will contradict that $\bd$ is \tcgl{}.

Define $v : [s_1,s_2] \rightarrow \mathbb{R}^2$ by
\[
v(s) = \begin{cases}
\lim_{t\downarrow s_1} \gamma'(s),& s=s_1,\\
\lim_{t\uparrow s} \gamma'(s),& s \in (s_1, s_2].
\end{cases}
\]
Note that $v(s)$ is in the tangent cone of $\bd$ at $\gamma(s)$ so that $\bd\cap\disk{\gamma(s)}{r}$ is graph-like using the orientation given by $v(s)$.

Define $\phi(s) : [s_1, s_2] \rightarrow \mathbb{R}$ by $\phi(s) = \langle v(s), p - \gamma(s)\rangle$.
Note that from $\gamma(s_1)$ both $v(s_1)$ and $p - \gamma(s_1)$ are directions pointing into the circle so $\phi(s_1) > 0$.
Similarly, $v(s_2)$ points out and $p - \gamma(s_2)$ points in so that $\phi(s_2) < 0$.

Observe that $v$ (and therefore $\phi$) is piecewise continuous since $\gamma$ is piecewise $C^1$.
By a piecewise continuous analogue of the intermediate value theorem, there exists $\bar{s} \in [s_1, s_2]$ such that
\[
\lim_{t\rightarrow \bar{s}^-} \phi(t) \leq 0 \leq \lim_{t\rightarrow \bar{s}^+} \phi(t).
\]

By continuity of the inner product and $\gamma$, we have
\[
\lim_{t\rightarrow \bar{s}^-} \phi(t) %
= \langle \lim_{t\rightarrow \bar{s}^-} \gamma'(t), p-\gamma(\bar{s})\rangle.
\]
Similarly, $\lim_{t\rightarrow \bar{s}^+} \phi(t) = \langle \lim_{t\rightarrow \bar{s}^+} \gamma'(t), p-\gamma(\bar{s})\rangle$

If $\gamma$ is differentiable at $\bar{s}$, then $\phi(\bar{s}) = \lim_{t \rightarrow\bar{s}}\phi(t) = 0$ and we have our contradiction.
Otherwise, let $w_1 = \lim_{t\rightarrow \bar{s}^-} \gamma'(t)$ and $w_2 = \lim_{t\rightarrow \bar{s}^+} \gamma'(t)$.
As both $w_1$ and $w_2$ are in the convex tangent cone of $\bd$ at $\gamma(\bar{s})$, any positive linear combination of them is as well.
Letting $\psi(\lambda) = \lambda w_1 + (1-\lambda)w_2$, we have 
\[
\langle \psi(0), p - \gamma(\bar{s})\rangle \leq 0 \leq \langle \psi(1), p-\gamma(\bar{s})\rangle.
\]
Noting that $\psi$ is continuous in $\lambda$, we apply the intermediate value theorem to obtain $\bar{\lambda} \in (0,1)$ such that $\langle \psi(\bar{\lambda}), p-\gamma(\bar{s})\rangle = 0$.
Letting $w = \psi(\bar{\lambda})$, we obtain our contradiction.
\end{enumerate}
Therefore, there are no other points of intersection and $|\bd \cap \cir{p}{r}| = 2$.
\end{proof}
\end{thm}

\begin{dfn}
We say that $\shape$ has the two-arc property for a given radius $r$ if for every point $p \in \bd$, we have that $\disk{p}{r}$ divides $\bd$ into two connected arcs: $\bd \cap \disk{p}{r}$ and $\bd \backslash \disk{p}{r}$.
Instead of considering how $\disk{p}{r}$ divides $\bd$, we can equivalently frame the definition in terms of how $\bd$ divides $\cir{p}{r}$.
That is, $\shape$ has the two-arc property if the circle $\cir{p}{r}$ is divided into two connected arcs by $\bd$ for every $p \in \bd$.  \end{dfn}

\begin{cor}
If $\shape$ is \tcgl{} for some radius $r$, then it has the two-arc property.
\label{lem:tcgl-2arc}
\begin{proof}
This is a trivial consequence of Theorem \ref{thm:two-intersections}.
\end{proof}
\end{cor}

\begin{cor}
If $\shape$ is \tgl{} for some radius $r$, then it has the two-arc property for radius $r$.
\label{cor:tgl-2arc}
\end{cor}

\begin{figure}[htp!]
  \centering
  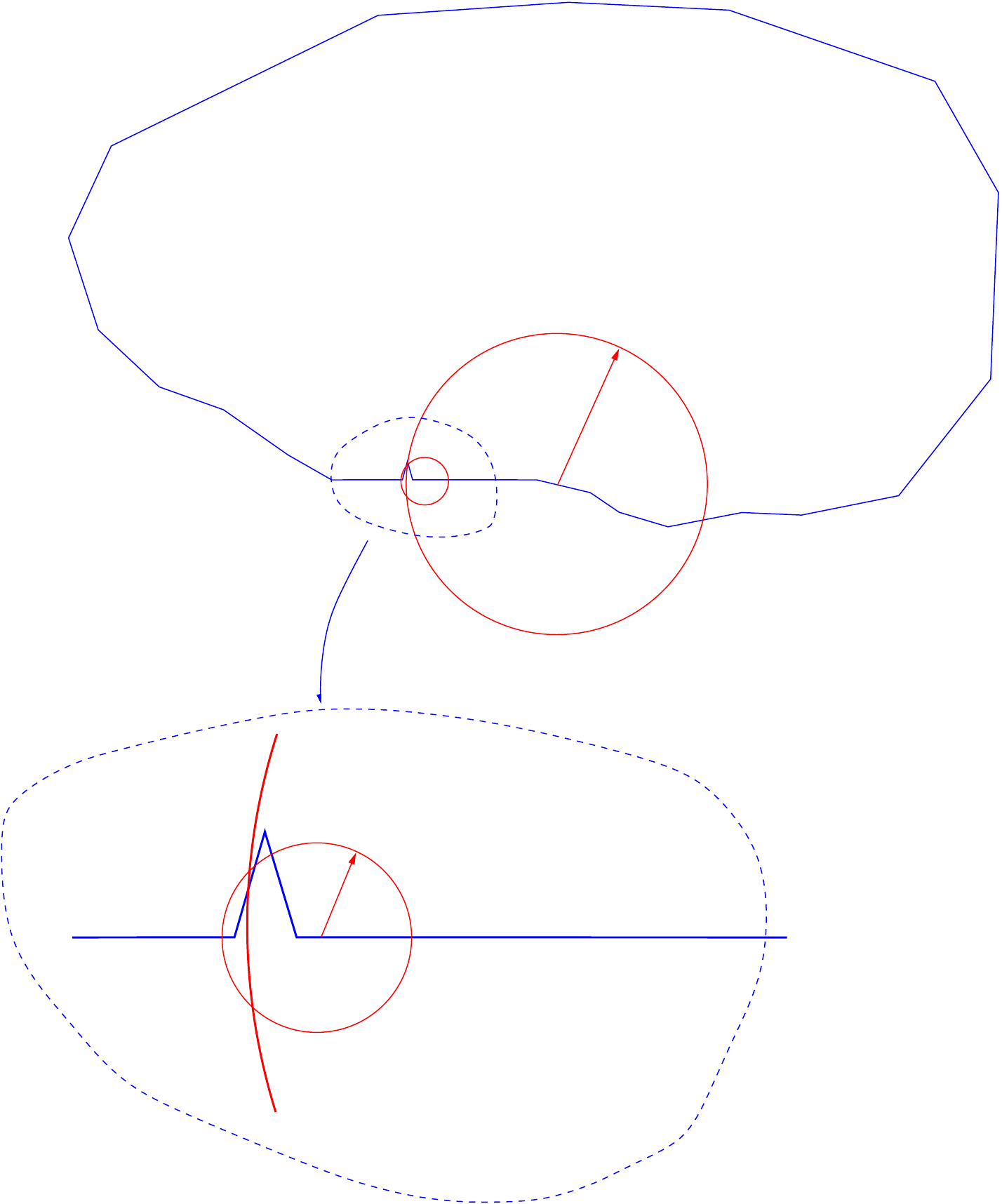
  \caption{The two-arc property for $r = \hat{r}$
    does not imply that it holds for all $r < \hat{r}$}
 \label{fig:tgl_counter_ex1}
\end{figure}
\begin{rem}
\label{rem:2arc}
  While the assumption of the two-arc property for disks of radius $r =
  \hat{r}$ does not imply the two-arc property for all $r < \hat{r}$
  (see Figure~\ref{fig:tgl_counter_ex1}), it is the case that TGL for
  $r = \hat{r}$ does imply that $\gamma$ is TGL for all $0 < r <
  \hat{r}$. The fact that $\gamma$ is TGL for all $0 < r < \hat{r}$ follows
  easily from the definition of TGL and the fact that $\disk{p}{r} \subsetneq \disk{p}{\hat{r}}$.
\end{rem}

\subsection{Notation}
\begin{figure}[htp!]
  \centering
  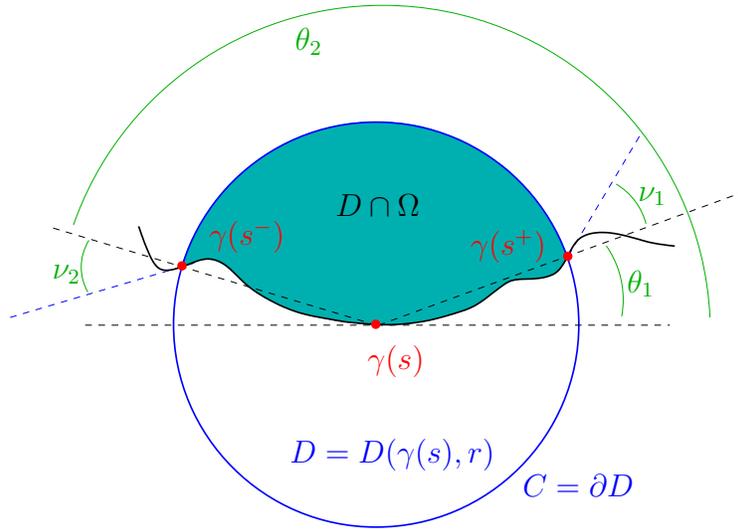
  \caption{Notation and conventions}
\label{fig:step0}
\label{fig:notation}
\end{figure}

Suppose that $\bd$ is \tcgl{} with radius $r$ and we have some $s \in [0, L)$ such that $\bd$ is \tgl{} at $\gamma(s)$ with radius $r$.
Since $\bd$ is TGL at $\gamma(s)$, it has two points of intersection with $\cir{\gamma(s)}{r}$ by theorem \ref{thm:two-intersections}.
In the orientation forced by the TGL condition, one of these points of intersection must be on the right side of the circle and one must be on the left side.

With reference to figure \ref{fig:step0} we define $s^{+}(s)$ and $s^{-}(s) \in [0,L)$ so that $\gamma(s^{+}(s))$ is the point of intersection on the right and $\gamma(s^{-}(s))$ is the point of intersection on the left.
The notation is motivated by the fact that $0 < s^{-}(s) < s < s^{+}(s) < L$ in general due to our convention that $\gamma$ traverses $\bd$ counterclockwise.
The only case where this is not true is when $\gamma(L) = \gamma(0)$ is in the disk but even then it will hold for a suitably shifted $\hat{\gamma}$ that starts at some point outside the current disk.

The quantities $\theta_1(s)$ and $\theta_2(s)$ are the angles that the rays from the origin to the right and left points of intersection, respectively, make with the positive $x$ axis.
We can assume $\theta_1(s) \in (-\frac{\pi}{2},\frac{\pi}{2})$ and $\theta_2(s) \in (\frac{\pi}{2},\frac{3\pi}{2})$.

We define $\nu_1(s)$ as the angle between the vector $\gamma(s^{+}(s))-\gamma(s)$ and the vector $\lim_{t\downarrow s^{+}(s)}\gamma'(t)$, the one-sided tangent to $\bd$ at the point of intersection on the right.
That is, we are measuring the angle between the outward normal to the disk at the point of intersection and the actual direction $\gamma$ is going as it exits the disk.
We define $\nu_2(s)$ similarly.
We have $\nu_1, \nu_2 \in (-\frac{\pi}{2},\frac{\pi}{2})$ due to the fact that all circle crossings are transverse by theorem \ref{thm:two-intersections}.

When the proper $s$ to use is implied by context, we will often simply write $s^{+}$, $s^{-}$, $\theta_1$, $\theta_2$, $\nu_1$ and $\nu_2$ in place of $s^{+}(s)$, $s^{-}(s)$, and so forth.
\subsection{Calculus on Tangent Cones}

The following result is a version of the intermediate value theorem for elements of the tangent cones.
\begin{lem}
\label{lem:intvalue}
Suppose $\bd$ is \tcgl{} on $\disk{\gamma(s)}{r}$ and $s_1 < s_2$ such that $\gamma(s_1), \gamma(s_2) \in \disk{\gamma(s)}{r}$.
Further suppose that $w_1 \in \tc{\gamma}{s_1}$, $w_2 \in \tc{\gamma}{s_2}$, $\alpha \in (0, 1)$, and let $w' = \alpha w_1 + (1-\alpha)w_2$.
Then, there exists $s' \in [s_1, s_2]$ such that either $w'$ or $-w'$ is in $\tc{\gamma}{s'}$.
\begin{proof}
Let $n$ be a unit vector in $\mathbb{R}^2$ with $n \perp (\alpha w_1 + (1-\alpha)w_2)$.
We have $\alpha\langle n, w_1 \rangle = -(1-\alpha)\langle n, w_2 \rangle$.
It suffices to consider only $\langle n, w_1 \rangle \leq 0 \leq \langle n, w_2 \rangle$ as the argument is identical in the other case.
Note that since $0 \leq \langle n, w_2 \rangle = c_1\langle n, \lim_{t\uparrow s_2}\gamma'(t)\rangle + c_2\langle n, \lim_{t\downarrow s_2}\gamma'(t)\rangle$ for some nonnegative constants $c_1, c_2$ not both zero, at least one of the inner products on the right is nonnegative.
Using the notation of definition \ref{def:tcone}, we define $M_2 = \argmax_{\Gamma \in \{\Gamma^+, \Gamma^-\}} \langle n, \Gamma\rangle$ and have $\langle n, M_2 \rangle \geq 0$.
We similarly define $M_1$ with respect to $w_1$ such that $\langle n, M_1 \rangle \leq 0$.%

Define 
\[
v(t) = \begin{cases}
M_i, & t = s_i, i = 1, 2\\
\lim_{t\uparrow t} \gamma'(s)
\end{cases}
\]
and $\phi(t) = \langle n, v(t) \rangle$.
Since $\phi(s_1) \leq 0 \leq \phi(s_2)$, the argument proceeds as in theorem \ref{thm:two-intersections} to yield $\bar{s} \in [s_1, s_2]$ and $\bar{w} \in \tc{\gamma}{\bar{s}}$ such that $\langle n, \bar{w} \rangle = 0$.
Thus $\bar{w} = kw'$ for some $k \neq 0$.
In particular, $w' = \frac{1}{k}\bar{w}$ so either $w' \in \tc{\gamma}{\bar{s}}$ or $-w' \in \tc{\gamma}{\bar{s}}$ (depending on the sign of $k$).
\end{proof}
\end{lem}

In addition to the intermediate value theorem, we have an analogous mean value theorem for tangent cone elements.

\begin{lem}
\label{lem:meanvalue}
Suppose $\gamma:[a,b] \rightarrow \mathbb{R}^2$ is a simple, arc-length parameterized curve with piecewise continuous derivative defined on $(a,b)$ except possibly on finitely many points.
Further suppose that the image of $\gamma$ has no cusps.
Then there exists $c$ in $(a,b)$ such that either $\gamma(b)-\gamma(a)$ or $-(\gamma(b)-\gamma(a))$ is in $T_\gamma(c)$.
\begin{proof}
Let $n$ be a unit vector with $\langle \gamma(b)-\gamma(a), n \rangle = 0$.
Consider $\psi(t) = \langle \gamma(t) , n \rangle$ and note that $\psi'(t) = \langle \gamma'(t), n \rangle$ is defined wherever $\gamma(t)$ is differentiable.
We have $\int_a^b \psi'(t) = \psi(b) - \psi(a) = \langle \gamma(b) - \gamma(a), n \rangle = 0$.
Thus, either $\psi'(t) = 0$ everywhere it is defined or it takes on both positive and negative values.
In particular, there exists a point $c \in (a,b)$ such that either $\psi'(c) = 0$ or $\lim_{t\uparrow c} \psi'(t) \leq 0 \leq \lim_{t\downarrow c} \psi'(t)$.

If $\psi'(c) = 0$, then we have $\langle \gamma'(c), n \rangle = 0$ so that $\gamma'(c) = k(\phi(b) - \phi(a))$ for some $k \neq 0$.
As $\gamma'(c) \in \tc{\gamma}{c}$, we have $\frac{k}{|k|}(\phi(b) - \phi(a)) \in \tc{\gamma}{c}$ which gives us our conclusion.

If $\lim_{t\uparrow c} \psi'(t) \leq 0 \leq \lim_{t\downarrow c} \psi'(t)$, there exists $\alpha \in (0, 1)$ such that $0 = \alpha \lim_{t\uparrow c} \psi'(t) + (1-\alpha)\lim_{t\downarrow c} \psi'(t)$.
Note that $\lim_{t\uparrow c} \psi'(t) = \langle w_1, n \rangle$ and $\lim_{t\downarrow c} \psi'(t) = \langle w_2, n \rangle$ for some $w_1, w_2 \in \tc{\gamma}{c}$ and let $w' = \alpha{w_1} + (1-\alpha){w_2}$.

By the convexity of $\tc{\gamma}{c}$, we have $w' \in \tc{\gamma}{c}$ with $\langle w', n\rangle = 0$ which follows as in the previous case.
\end{proof}
\end{lem}

The following lemma tells us that the \tcgl{} condition is sufficient to apply lemma \ref{lem:meanvalue}.

\begin{lem}
If $\bd$ is \tcgl{} for some radius $r$, then $\bd$ has no cusps.
\begin{proof}
Suppose $\bd$ has a cusp at $\gamma(s)$.
Then, using the terminology of definition \ref{def:tcone} and the fact that $\gamma$ is arc length parameterized, we have $\Gamma^+ = -\Gamma^-$.
We let $w = 0$ and note that $w = \Gamma^+ + \Gamma^- \in \tc{\gamma}{s}$.
Letting $u,v\in\bd\cap\disk{\gamma(s)}{r}$ with $u \neq v$, we have $\langle w, u-v\rangle = 0$, contradicting the fact that $\bd$ is \tcgl{}.
Therefore, $\bd$ has no cusps.
\end{proof}
\end{lem}

\subsection{TCGL Boundary Properties}

The following technical lemmas allow us to bound various distances and areas encountered in \tcgl{} boundaries.

\begin{lem}
Suppose that $\bd$ is \tcgl{} with radius $r$ and points $p_1, p_2 \in \bd$ with $d(p_1, p_2) < r$.
Then one of the arcs (call it $P$) along $\bd$ between $p_1$ and $p_2$ is such that, for any two points $q_1, q_2 \in P$, we have $d(q_1, q_2) < r$.
\begin{proof}
Note that $p_2 \in \disk{p_1}{r}$ so that there is an arc along $\bd$ from $p_1$ to $p_2$ which is fully contained in the interior of $\disk{p_1}{r}$ by theorem \ref{thm:two-intersections}.
We will call this arc $P$.

For all $x$ on $P$, let $P_x$ denote the subpath of $P$ from $p_1$ to $x$ (so $P = P_{p_2}$).
We claim that $P_x$ is contained in $\disk{x}{r}$ for all $x$ on $P$ (thus, $P$ is contained in $\disk{p_2}{r}$).
Indeed, if this were not the case, then there must be some $\hat{x}$ on $P$ such that $P_{\hat{x}}$ is contained in $\disk{\hat{x}}{r}$ but $\cir{\hat{x}}{r} \cap P_{\hat{x}}$ is nonempty (i.e., we can move the disk along $P$ until some part of the subpath hits the boundary).
That is, the subpath $P_{\hat{x}}$ has a tangency with the disk $\disk{\hat{x}}{r}$ which is impossible because of theorem \ref{thm:two-intersections}.

Let $q_1 \in P$ and note that since $P_x$ is contained in $\disk{x}{r}$ for all $x$ on $P$, we have that $P$ is contained in $\disk{q_1}{r}$.
Therefore, $d(q_1, q_2) < r$ for all $q_1, q_2 \in P$ as desired.
\end{proof}
\end{lem}

\begin{lem}
\label{lem:arcbound}
If $q_1 = \gamma(s_1), q_2 = \gamma(s_2) \in P$ where $P$ is as in the previous lemma, then the arc length between $q_1$ and $q_2$ along $P$ is at most $\sqrt{2}d(q_1, q_2)$.
\begin{proof}
Since $\shape$ is \tgl{}, for any $w_1 \in \tc{\gamma}{s_1}, w_2 \in \tc{\gamma}{s_2}$, the angle between $w_1$ and $w_2$ is at most $\frac{\pi}{2}$.
Since this is true for all $q \in P$, there is a point $q' = \gamma(s') \in P$ and $w' \in \tc{\gamma}{s'}$ such that the angle between $w'$ and tangent vectors for any other point $q \in P$ is at most $\frac{\pi}{4}$.

This means that $P$ is the graph of a Lipschitz function $g$ of rank 1 in the orientation defined by $w'$.
This does not necessarily imply that $\disk{q'}{r} \cap \bd$, $\disk{p_1}{r}\cap \bd$ or $\disk{p_2}{r}\cap \bd$ is the graph of a Lipschitz function; we explore a Lipschitz condition for the disks in section \ref{sec:tgl-ncgl}.
Let $x_1, x_2 \in [-r, r]$ with $p_1 = (x_1, g(x_1))$, $p_2 = (x_2, g(x_2))$.
Then the arclength from $p_1$ to $p_2$ is given by 
\[
\int_{x_1}^{x_2} \sqrt{1 + g'(x)^2}\,dx \leq \int_{x_1}^{x_2} \sqrt{2}\,dx = \sqrt{2}(x_2-x_1) \leq \sqrt{2}d(p_1, p_2).\qedhere
\]
\end{proof}
\end{lem}

\begin{lem}
\label{lem:areabound}
If $\gamma$ is \tcgl{} with radius $r$ and $0 \leq s_1 \leq s_2 < L$ with $d(\gamma(s_1), \gamma(s_2)) = \delta < r$, then the image of $[s_1, s_2]$ together with the straight line from $\gamma(s_1)$ to $\gamma(s_2)$ enclose a region with $O(\delta^2)$ area.
\begin{proof}
By Lemma \ref{lem:arcbound}, we have that the image of $[s_1, s_2]$ under $\gamma$ has arc length $s_2 - s_1 \leq \sqrt{2}\delta$.
Therefore, the region of interest has perimeter at most $(\sqrt{2}+1)\delta$ so by the isoperimetric inequality has area at most $\frac{(\sqrt{2}+1)^2}{4\pi}\delta^2$ from which the conclusion follows.
\end{proof}
\end{lem}

\section{TCGL polygonal approximations}
\label{sec:tgl-ncgl}

If $\shape$ is \tcgl{} with radius $r$, it can sometimes be nice to know that there is an approximating polygon to $\shape$ which is also \tcgl{}.
The following lemmas explore this idea.

\begin{lem}
If $\bd$ is TCGL with radius $r$ then for each $\epsilon \in (0, r)$, then there exists a polygonal approximation to $\bd$ that is TCGL with radius $r - \epsilon$ and such that every point on $\bd$ is within distance $\frac{\epsilon}{6}$ of the polygon.
\label{lem:tcgl-poly}
\begin{proof}
First, choose a finite number of points along the boundary such that the arc length along $\gamma$ between any two neighboring points is no more than $\frac{\epsilon}{3}$.
These will be the vertices of our polygon.
Similarly to $\gamma$, we let $\phi$ be an arclength parameterization of this polygon so that they both encounter their common points in the same order.

The fine spacing between vertices guarantees that we obtain the $\frac{\epsilon}{6}$ bound.
Indeed, given any point $p \in \bd$ and its neighboring vertices $v_1$ and $v_2$, the arc length along $\bd$ from $v_1$ to $p$ plus that from $p$ to $v_2$ is at most $\frac{\epsilon}{3}$ by assumption.
Since Euclidean distance is bounded above by arc length, we have $d(p, v_1) + d(p, v_2) \leq \frac{\epsilon}{3}$.
This bound in turn implies that at least one of $d(p,v_1)$ and $d(p,v_2)$ is bounded above by $\frac{\epsilon}{6}$.

Consider a point $p = \phi(t)$ on a side of the polygon (i.e., not a vertex) and its neighboring vertices $v_1 = \phi(t_1) = \gamma(s_1)$ and $v_2 = \phi(t_2) = \gamma(s_2)$ (chosen with $t_1 < t < t_2$ and $s_1 < s_2$).
By lemma \ref{lem:meanvalue}, there exists $s \in (s_1, s_2)$ such that $v_2 - v_1 \in T_\gamma(s)$.
Note that this is the only member of $T_\phi(t)$ up to positive scalar multiplication.

Combining the arcs along $\gamma$ and $\phi$ between $v_1$ and $v_2$, we obtain a closed curve with total length at most $\frac{2\epsilon}{3}$, so that the distance between any two points on the curve is at most $\frac{\epsilon}{3}$.
That is, for any $s' \in [s_1, s_2]$ and $t' \in [t_1, t_2]$, we have $d(\gamma(s'), \phi(t')) \leq \frac{\epsilon}{3}$.

Let $x \in \disk{\phi(t)}{r-\epsilon}$.
Then $d(x, \gamma(s)) \leq d(x, \phi(t)) + d(\phi(t), \gamma(s)) \leq r - \frac{2\epsilon}{3}$ so that $\disk{\phi(t)}{r-\epsilon}$ is contained in $\disk{\gamma(s)}{r - \frac{2\epsilon}{3}}$.

Let $a, b$ be distinct points on the polygon in $\disk{\phi(t)}{r-\epsilon}$ and consider the line connecting them.
This line also intersects $a', b'$ on $\gamma$ such that we have $a'\neq b'$, $d(a, a') \leq \frac{\epsilon}{3}$ and $d(b,b') \leq \frac{\epsilon}{3}$ so that $a', b' \in \bd \cap \disk{\gamma(s)}{r}$.
As $a - b = c(a'-b')$ for some scalar $c > 0$, we have
\[
\langle v_2 - v_1, a - b \rangle = c\langle v_2 - v_1, a' - b' \rangle \neq 0
\]
since $\gamma$ is TCGL at $\gamma(s)$ with radius $r$ and $v_2 - v_1 \in T_\gamma(s)$.
Thus $\phi$ is TCGL at $p$ with radius $r - \epsilon$.

The case where $p = \phi(t)$ is a vertex is similar but we must consider an arbitrary vector $w \in T_\phi(t)$ in the inner product.
We wish to show that, for every $w \in T_\phi(t)$, there is a $s'$ such that either $w$ or $-w \in T_\gamma(s')$ and $d(p, \gamma(s')) \leq \frac{\epsilon}{3}$, after which the proof follows as in the first case with $w$ (or $-w$) in place of $v_2-v_1$.
We let $\gamma(s) = \phi(t) = p$ and let $v_1 = \phi(t_1) = \gamma(s_1)$ and $v_2 = \phi(t_2) = \gamma(s_2)$ be the neighboring vertices (so $t_1 < t < t_2$ and $s_1 < s < s_2$).

As above, there exist $s_1', s_2'$ such that $s_1 \leq s_1' \leq s \leq s_2' \leq s_2$, $\gamma(s) - \gamma(s_1) \in \tc{\gamma}{s_1'}$ and $\gamma(s_2) - \gamma(s) \in \tc{\gamma}{s_2'}$.
Note that $\tc{\phi}{t}$ is exactly the set of positive linear combinations of these vectors.
By lemma \ref{lem:intvalue}, for every $w \in \tc{\phi}{t}$, there is a $s' \in [s_1', s_2']$ such that $w \in \tc{\gamma}{s'}$.
As $d(p, \gamma(s')) < \frac{\epsilon}{3}$, the proof is complete.
\end{proof}
\end{lem}
\begin{dfn}
We say that $\shape$ is tangentially graph-like and Lipschitz (TGLL) with radius $r$ if $\shape$ is \tgl{} with radius $r$ and there is some constant $0 < K < \infty$ such that for every $p \in \bd$, the arc $\disk{p}{r} \cap \bd$ is the graph of a Lipschitz function (in the same orientation used by the \tgl{} definition) and that the Lipschitz constant is at most $K$.
\end{dfn}

\begin{rem}
Note that \tgl{} does not imply tangentially graph-like and Lipschitz: taking
$\gamma$ to be a square with side length $5$ whose corners are
replaced by quarter circles of radius $1$ and then considering disks
of radius $\sqrt{2}$ centered on $\gamma$ yields one example.
\end{rem}

Because $\gamma$ is arclength parameterized by $s$, $||\dgm{s}|| = 1$
for all $s$.  Since $\gamma$ is assumed $C^1$ on its compact domain
$[0,L]$, $\dgmm$ is uniformly continuous: for any $\epsilon > 0$,
there is a $\delep$ such that if $|s_2-s_1| < \delep$ then
$||\dgm{s_2} - \dgm{s_1}||< \epsilon$.

We will use the fact that $\gamma$ always crosses $\partial D$ transversely to prove that $\gamma$ is in fact TGLL on slightly
bigger disks of radius $r+\delta$ as long as one takes a somewhat
bigger Lipschitz constant $\hat{K}$.
It is then an immediate result of lemma \ref{lem:tcgl-poly} that we can find an approximating polygon that is TCGL with radius $r$.

\begin{figure}[htp!]
  \centering
  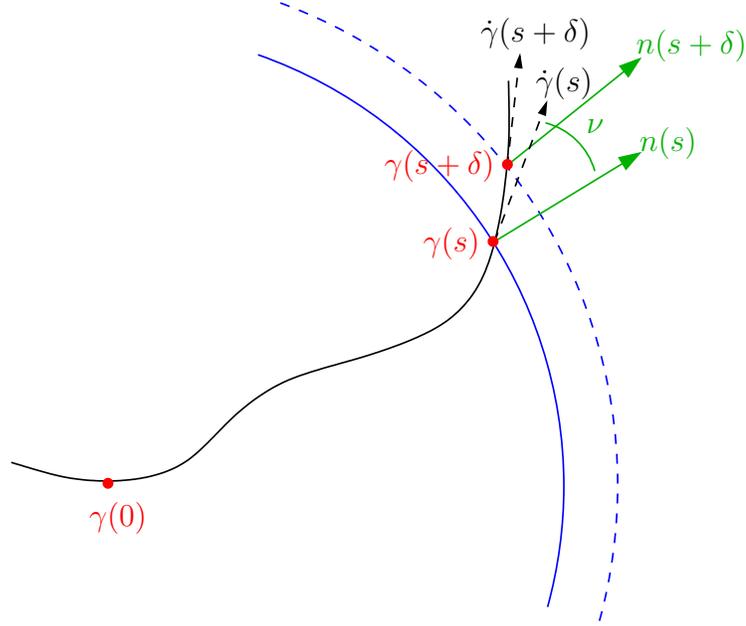
  \caption{TGLL implies TCGL: Step one}
 \label{fig:step1}
\end{figure}

\begin{lem}
  If $\gamma$ is TGLL with radius $r$, then it is TGLL with radius $r+\delta$ for some $\delta > 0$ and there is an 
  approximating polygon $P_{\gamma}$ which is TCGL with radius $r$.
\end{lem}

\begin{proof}
\noindent{\bf Step 1: Show that the quantities $\nu_1$ and $\nu_2$
are continuous as a function of $s \in [0,L]$.(see Fig.~\ref{fig:step0})}

 Define $R^2(s,t)
\equiv ||\gm{s} - \gm{t}||^2$.  Taking the derivative, we get
\[DR = \left[ \left\langle \frac{\gm{s} - \gm{t}}{R(s,t)} ,
\dgm{s}\right\rangle , \left\langle \frac{\gm{t}-\gm{s}}{R(s,t)},
\dgm{t}\right\rangle \right]. \] Because $\nu_1$ and $\nu_2$ are both less than
$\pi/2$ and $\gamma$ is graph-like in the disk, we have that both
elements of this derivative are nowhere zero. By the implicit function
theorem, we get that $s^{-}(s)$ and $s^{+}(s)$ are continuous
functions of $s$. From this it follows that $\nu_1$ and $\nu_2$ are
continuous on $[0,L]$.

\noindent{\bf Step 2:} From the previous step and the compactness of $[0,L]$ we
get that $\nu_1(s)$ and $\nu_2(s)$ are both bounded by $M_\nu <
\pi/2$.
We define $\enu \equiv \pi/2 - M_{\nu} > 0$.
Fix a $t\in[0,L]$. Define $\rhh(s)$ by $\rhh^2(s) = R^2(s,t) =
||\gm{s} - \gm{t}||^2$. Then $\dot{\rhh}(s) = \langle \frac{\gm{s}
  -\gm{t}}{\rhh}, \dgm{s}\rangle = \langle n_t(s), \dgm{s}\rangle $
where $n_t(s)= \frac{\gm{s}-\gm{t}}{||\gm{s}-\gm{t} ||} =
\frac{\gm{s}-\gm{t}}{\rhh}$, the external normal to $\partial
D(\gamma(t),\rhh)$ at $\gm{s}$ (see Figure \ref{fig:step1}).  On any interval in $s$ where
$\dot{\rhh}(s) > 0$ we have that $\rhh(s)$ is one to one and strictly increasing.
Define $s^{*} \equiv s^{+}(t)$ and $s_{*} \equiv s^{-}(t)$. We showed
above that $\dot{\rhh}(s^*) = \langle n_t(s^*), \dgm{s^*}\rangle \geq
\cos(M_{\nu}) > 0$.

For $\langle n_t(s),\dgm{s}\rangle = 0$, $n_t(s)$ and $\dgmm$ will
have to have together turned by at least $\pi/2 - M_{\nu}$ radians.
And until they have turned this far, $\langle n_t(s),\dgm{s}\rangle >
0$. But $\dot{n}_t(s) \leq \frac{1}{\rho} \leq \frac{1}{r_{min}}$ for
some $r_{min} > 0$. (Choosing $r_{min} = \frac{r}{2}$ works.) And
$\dgmm$ is uniformly continuous on $[0,L]$. Therefore, there is a
$\delta_s$ such that on $[s^* , s^* + \delta_s]$, $n_t(s)$ and $\dgmm$
both turn by less than $\enu/3$.  Therefore, for $s \in [s^* , s^* +
\delta_s]$, we have that $\langle n_t(s), \dgm{s}\rangle > \cos(\pi/2
- \enu/3)$ and $\gamma([s^* , s^* + \delta_s))$ intersects $C =
\partial D(\gm{t},\rho)$ once for each $\rho \in [r,r+\delta_r]$, where
$\delta_r \equiv \delta_s\cos(\pi/2 - \enu/3)$.

A completely analogous argument works to show that $\gamma([s_*
-\delta_s, s_*])$ intersects $C = \partial D(\gm{t},\rho)$ once for
each $\rho \in [r,r+\delta_r]$. 

Define $d(t)$ to be the distance from $D(\gm{t},r)$ to
$\gamma\setminus\gamma([s_*-\delta_s,s^*+\delta_s])$. Since $\gamma$
is TGL, d(t) is greater than zero for all $t$ and is continuous in
$t$. Therefore, there is a smallest distance $\delta_d$ such that
$d(t) \geq \delta_d$ for all $t$. Define $\delta_{\gamma_o} =
\min(\delta_d/2,\delta_r/2)$.

Therefore, $\partial D(\gm{t},\rho)$ intersects $\gamma$ exactly
twice for $\rho \in [r,r+\delta_{\gamma_o}]$ for any $t\in[0,L]$. 

A similar argument shows that $\partial D(\gm{t},\rho)$ intersects
$\gamma$ exactly twice for $\rho \in [r-\delta_{\gamma_i},r]$ for any
$t\in[0,L]$. Defining $\delta_\gamma \equiv \min(\delta_{\gamma_i},
\delta_{\gamma_o})$ we get that $\partial D(\gm{t},\rho)$ intersects
$\gamma$ exactly twice for $\rho \in [
r-\delta_{\gamma},r+\delta_{\gamma} ]$, with the additional fact that
$\langle n_t(s), \dgm{s}\rangle > \cos(\pi/2
- \enu/3)$ at all those intersections.

\noindent{\bf Step 3:} TGLL implies that there is a constant $K <\infty$
such that $\gamma\cap\D(\gm{t},r)$ is the graph of a function whose
x-axis direction is parallel to $\dgm{t}$ and this function is
Lipschitz with Lipschitz constant $K$.

Since $\dgmm$ is uniformly continuous, there will be a $\delta_1$ such
that if $|u-v| < \delta_1$, then $\angle(\dgm{u}, \dgm{v}) <
\arctan{2K} - \arctan{K}$. Define $\delta_{K,s} = \min(\delta_s,
\delta_1)$. Define $\delta_{K,r} = \min(\delta_\gamma,
\delta_{K,s}\cos(\pi/2-\enu/3))$. Then
$\gamma\cap\D(\gm{t},r+\delta_{K,r})$ is the graph of a Lipschitz
function with Lipschitz constant at most $2K$ when $\dgm{t}$ is used
as the x-axis direction. That is, for all $t$, $\gamma$ is TGLL with
Lipschitz constant $2K$ for disks of radius $r+\delta_{K,r}$.
The result follows by lemma \ref{lem:tcgl-poly}.
\end{proof}

\section{Derivatives of $g(s,r)$}
\label{sec:derivatives}

\begin{figure}[htp!]
  \begin{center}
    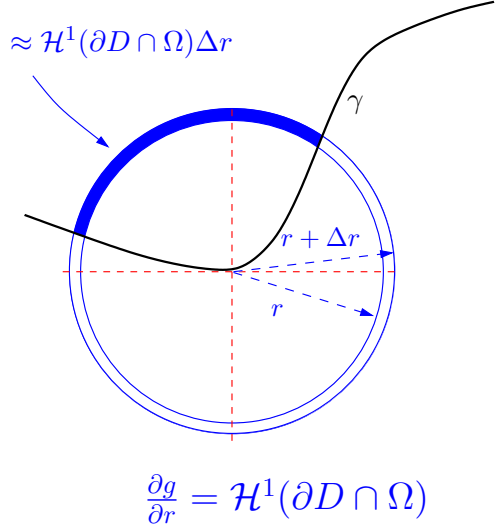
  \end{center}
\caption{Deriving $\pd{g}{r}$ as the arclength of the circular segment.}
\label{fig:dgdr}
\end{figure}

\begin{lem}
\label{lem:dr}
Using the notation of figure \ref{fig:step0}, we have $\pd{}{r}g(s, r) = (\theta_2 - \theta_1)r$.
That is, the derivative exists and equals the length of the curve $\cir{\gamma(s)}{r} \cap \shape$.
\begin{proof}
We have (see figure~\ref{fig:dgdr})
\[
\pd{}{r}g(s,r) 
= \lim_{\Delta r \rightarrow 0}\frac{\area (\shape \cap \disk{\gamma(s)}{r+\Delta r})  - \area (\shape \cap \disk{\gamma(s)}{r})}{\Delta r}.
\]

This difference of areas can be modeled by the difference in the circular sectors of $\disk{\gamma(s)}{r+\Delta r}$ and $\disk{\gamma(s)}{r}$ with angle $\theta_1 - \theta_2$.
The actual area depends on the image of $\gamma$ outside of $\disk{\gamma(s)}{r}$, but this correction will be a subset of the circular segment of $\disk{\gamma(s)}{r+\Delta r}$ which is tangent to $\disk{\gamma(s)}{r}$ at the point $\gamma$ exits.
This has area $O(\Delta r^2)$ by lemma \ref{lem:areabound}.

Thus we have
\[
\pd{}{r}g(s,r) = \lim_{\Delta r \rightarrow 0} \frac{(\theta_1 - \theta_2)r\Delta r + \frac{1}{2}(\theta_1 - \theta_2)\Delta r^2 + O(\Delta r^2)}{\Delta r} = (\theta_1 - \theta_2)r.\qedhere
\]
\end{proof}
\end{lem}

\begin{figure}[htp!]
  \begin{center}
    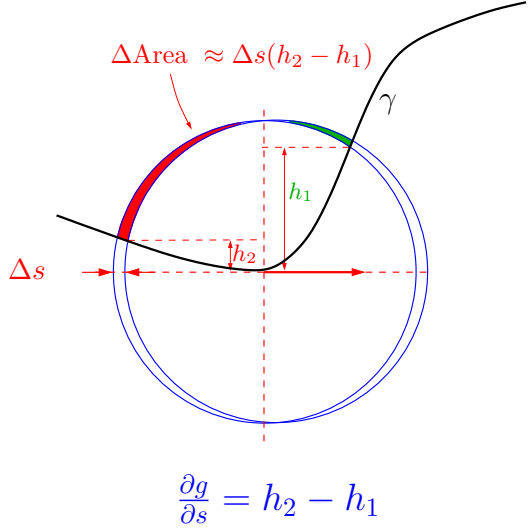
  \end{center}
\caption{Deriving $\pd{g}{s}$ as the difference in heights of the entry and exit points}
\label{fig:dgds}
\end{figure}
\begin{lem}
Using the notation of figures \ref{fig:step0} and \ref{fig:dgds}, we have $\pd{}{s}g(s,r) = h_2 - h_1 = r\sin(\theta_2) - r\sin(\theta_1)$.
\label{lem:ds}
\begin{proof}
We have
\[
\pd{}{s} g(s,r) = \lim_{\Delta s \rightarrow 0} 
\frac{\area (\shape \cap \disk{\gamma(s+\Delta s)}{r})  - \area (\shape \cap \disk{\gamma(s)}{r})}{\Delta s}.
\]
The situation is illustrated in figure~\ref{fig:dgds} where we can see that the area being added as we go from $s$ to $s + \Delta s$ is the shaded region on the right with height $r - h_1$ and, considering first-order terms only, uniform width $\Delta s$ so has area $(r-h_1)\Delta s$.
Similarly, we are subtracting the area $(r-h_2)\Delta s$ on the left.
Therefore, we have
\[
\pd{}{s}g(s,r) = \lim_{\Delta s\rightarrow 0} \frac{(r-h_1)\Delta s - (r-h_2)\Delta s}{\Delta s} = h_2 - h_1.
\]
\end{proof}
\end{lem}

\section{Reconstructing shapes from T-like data}
\label{sec:reconstruct-T}

\begin{figure}[htp!]
  \centering
  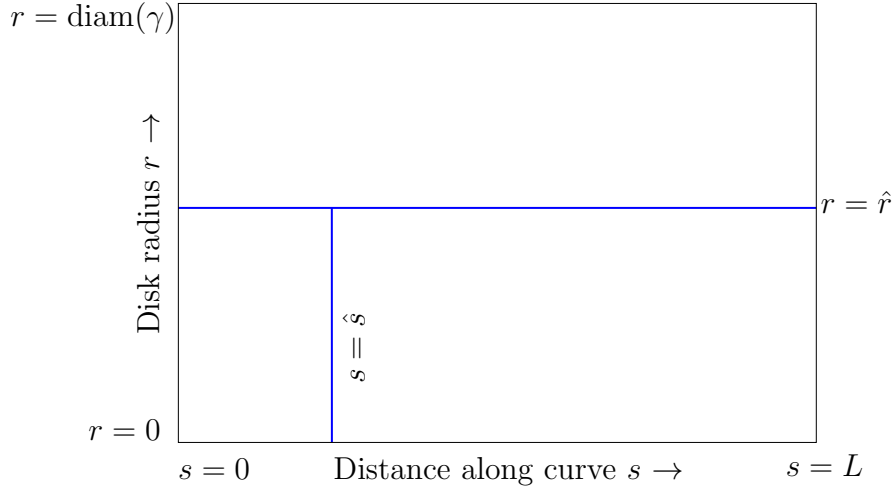
  \caption{T-like data: we restrict the domain of $g(s,r)$ to a fixed radius
    $\hat{r}$ plus any vertical segment from $r=0$ to $r = \hat{r}$}
 \label{fig:t_like_data1}
\end{figure}

In this section, we consider the case where nonasymptotic densities and first derivatives are
known along a T-shaped set (i.e., for all $s$ with a fixed radius $\hat{r}$ and for all $r \leq \hat{r}$ with a fixed $\hat{s}$).
We show that this information is sufficient to guarantee reconstructability modulo reparametrizations, translations, and rotations.

\begin{lem}
  Assume that $\gamma$ is TGL for $\hat{r}$ (and thus all $r \leq \hat{r}$). Then if we know
  $g(s,r)$, $g_s(s,r) = \pd{g(s,r)}{s}$, and $g_r(s,r) =
  \pd{g(s,r)}{r}$ for $(s,r)\in([0,L]\times\{\hat{r}\}) \cup
  (\{\hat{s}\}\times (0,\hat{r}])$, we can reconstruct
  $\gamma(s)\in\Bbb{R}^2$ for all $s\in[0,L]$ modulo reparametrizations, translation, and rotations. (See
  figure~\ref{fig:t_like_data1}.)
\end{lem}

\noindent \emph{Proof:} As was shown in section~\ref{sec:derivatives},
$g_r$ gives us the length of the arc $\partial D(s,\hat{r})\cap\Omega$
and $g_s$ tells us precisely what position this arc is along $\partial
D(s,\hat{r})$ with respect to the direction $\dgm{s}$. The assumption
of TGL for $r=\hat{r}$ implies TGL for $0<r<\hat{r}$ (see
remark~\ref{rem:2arc}) and this implies that $\gamma$ has the 2 arc
property and transverse intersections with $\partial \D(s,r)$ for all
disks corresponding to $(s,r)\in([0,L]\times\{\hat{r}\}) \cup
(\{\hat{s}\}\times [0,\hat{r}])$. Since we care only about
reconstructing a curve $\gamma$ isometric to the original curve, we
choose $\gm{\hat{s}} = (0,0)\in\Bbb{R}^2$ and $\dgm{\hat{s}} = (1,0)$.
Taken together, $g_s(\hat{s},r)$ and $g_r(\hat{s},r)$ locate both
points in $\partial D(\hat{s},r) \cap \gamma$ for all $r\in
[0,\hat{r}]$. This yields $\gamma\cap\D(\gm{\hat{s}},\hat{r})$. Now,
simply increase $s$, sliding the center of a disk of radius $\hat{r}$
along $\gamma\cap\D(\gm{\hat{s}},\hat{r})$, using $g_r(s,\hat{r})$ to
find the element of $\gamma\cap\D(\gm{s},\hat{r})$ outside
$D(\gm{\hat{s}},\hat{r})$, using the fact that the other element of
$\gamma\cap\D(\gm{s},\hat{r})$ is inside $\D(\gm{\hat{s}},\hat{r})$
and known. This process can be continued until the entire curve is
traced out in $\Bbb{R}^2$.

\section{TCGL Polygon Is Reconstructible from $g_r$ and $g_s$ without
  tail}
\label{sec:polygon-no-tail}

\begin{thm}
\label{thm:tcgl-poly-reconstruct}
For a \tcgl{} polygon $\shape$, knowing $g(s,r)$, $g_r(s,r)$ and $g_s(s,r)$ for all $s \in [0,L)$ and a particular $r$ for which $\bd$ is \tcgl{} is sufficient to completely determine $\shape$ up to translation and rotation; that is, we can recover the side lengths and angles of $\shape$.
\begin{proof}
For a given $s$ and $r$ where $g_r$ and $g_s$ exist, we can use them to obtain $r(\theta_2-\theta_1)$ as the length of the circular arc between the entry and exit points by Lemma \ref{lem:dr} and $r(\sin\theta_2-\sin\theta_1)$ as the difference in heights of the entry and exit points by Lemma \ref{lem:ds}.

We wish to recover $\theta_1$ and $\theta_2$ from these quantities.
Note that if $(\theta_1, \theta_2) = (\phi_1, \phi_2)$ is one possible solution, then so is $(\theta_1, \theta_2) = (2\pi - \phi_2, 2\pi - \phi_1)$ so solutions always come in pairs.

We can imagine placing a circular arc with angle $\frac{g_r}{r}$ on our circle and sliding it around until the endpoints have the appropriate height difference, yielding our $\theta_1$ and $\theta_2$.
Note that since $\shape$ is \tcgl{}, one endpoint must be on the left side of the circle and the other must be on the right and we cannot slide either endpoint to or beyond the vertical line through the center of the circle.

Therefore, as we slide the right endpoint down, the left endpoint slides up so that the height difference as a function of the slide is strictly monotonic.
Therefore, the slide that gives us $\theta_1$ and $\theta_2$ is unique for a given starting arc placement.
However, there are two starting arc placements: the first calls the angle for the right endpoint $\theta_1$ and the left endpoint $\theta_2$ (so the interior of $\shape$ is ``up'' in the circle) and the second swaps these (so the interior of $\shape$ is ``down'').
Since we have adopted the convention that $\bd$ is traversed in a counterclockwise direction (so the interior of $\shape$ is up in the circles) we therefore pick the first option; this gives us a unique solution for $\theta_1$ and $\theta_2$.

This procedure works whenever $g_r$ and $g_s$ exist which is certainly true whenever the density disk does not touch a vertex of $\shape$ either at its center or on its boundary because if we avoid these cases, then there is only one graph-like orientation to deal with and $\bd$ is $C^\infty$ for all the points that enter into the computation.
In fact, with a moment's thought, we can make a stronger statement than this: $g_r$ always exists and $g_s$ exists as long as the center of the density disk is not a vertex of the polygon.

We can identify the $s$ values at which $g_s(s,r)$ does not exist to obtain the arc length positions of the vertices (and therefore obtain side lengths).
For a given $s$ corresponding to a vertex, we can find $g_r$ and the one-sided derivatives $g_{s-}$ and $g_{s+}$.
These correspond to the graph-like orientations required by the polygon sides adjacent to the current vertex.

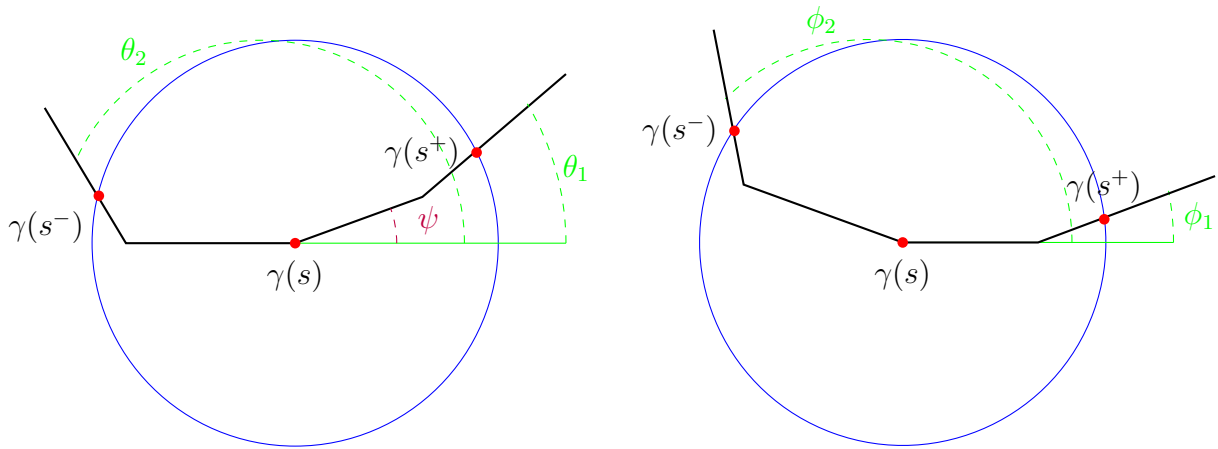
\begin{figure}
\makebox[\textwidth][c]{
\begin{tabular}{cc}
\begin{tikzpicture}[scale=0.9]
\draw[color=blue] (0,0) circle (3);

\draw[thick=2pt] (-3.7, 2) -- (-0:-2.5) -- (0,0) -- (20:2) -- (4,2.5);
\node[label=below left:{$\gamma(s^{-})$},style=circle,fill=red,minimum size=4pt,inner sep=0pt] (nenter) at (-2.9,.7) {};
\node[label=left:{$\gamma(s^{+})$},style=circle,fill=red,minimum size=4pt,inner sep=0pt] (nexit) at (2.68,1.34) {};
\node (nright) at (3,0) {};
\node (nleft) at (-3,0) {};

\draw[color=green] (0,0) -- (4,0);
\draw[dashed,color=green] (0:4) arc (0:31:4);
\node[color=green] () at (15:4.3) {$\theta_1$};
\draw[dashed,color=green] (0:2.5) arc (0:155:3);
\node[color=green] () at (130:3.7) {$\theta_2$};
\draw[dashed,color=purple] (0:1.5) arc (0:20:1.5);
\node[color=purple] () at (10:2) {$\psi$};
\node () at (270:3.23) {};
\node[label=below:{$\gamma(s)$},style=circle,fill=red,minimum size=4pt,inner sep=0pt] () at (0,0) {};

\end{tikzpicture}
&
\begin{tikzpicture}[scale=0.9,rotate=-20]
\draw[color=blue] (0,0) circle (3);
\draw[color=green] (200:0) -- (20:4);
\draw[dashed,color=green] (20:4) arc (20:31:4);
\node[color=green] () at (25:4.4) {$\phi_1$};
\draw[dashed,color=green] (20:2.5) arc (20:155:3);
\node[color=green] () at (130:3.5) {$\phi_2$};

\draw[thick=2pt] (-3.7, 2) -- (-0:-2.5) -- (0,0) -- (20:2) -- (4,2.5);
\node[label=left:{$\gamma(s^{-})$},style=circle,fill=red,minimum size=4pt,inner sep=0pt] (nenter) at (-2.9,.7) {};
\node[label=above:{$\gamma(s^{+})$},style=circle,fill=red,minimum size=4pt,inner sep=0pt] (nexit) at (2.68,1.34) {};
\node (nright) at (3,0) {};
\node (nleft) at (-3,0) {};
\node[label=below:{$\gamma(s)$},style=circle,fill=red,minimum size=4pt,inner sep=0pt] () at (0,0) {};

\end{tikzpicture}
\end{tabular}
}
\caption{Using $g_{s-}$ and $g_{s+}$ to obtain the polygon angle at $s$.}
\label{fig:tcgl-poly}
\end{figure}

Referring to Figure \ref{fig:tcgl-poly}, the one-sided derivatives along with the argument at the beginning of the proof yield the angles $\theta_1$, $\theta_2$, $\phi_1$, and $\phi_2$.
Thus we can calculate $\psi = \theta_1 - \phi_1$ which means that the polygon vertex at $s$ has angle $\pi - \psi$.

Doing this for all $s$ corresponding to vertices, we can determine all of the angles of the polygon.
With the side lengths identified earlier, this completely determines the polygon $\shape$ up to translation and rotation.
\end{proof}
\end{thm}

\section{Simple closed curves are generically reconstructible using fixed radius data}
\label{sec:generic}

We will assume that $\gamma$ is TGL for the radius $\hat{r}$.  We will also
assume that we know the first, second, and third derivatives of
$g(s,r)$ for $r=\hat{r}$.  Under these assumptions, $\gamma$ is
generically reconstructible. By generic we mean the \emph{admittedly weak}
condition of density -- reconstructible curves are $C^1$ dense in
the space of $C^2$ simple closed curves.

\begin{thm}
\label{thm:tcgl-reconstruct}
  Define $\Bbb{G} \equiv \{ \gamma| \gamma$ is a $C^2$ simple closed
  curve and TGL for $r=\hat{r}\}$.  Suppose that, for $r=\hat{r}$, for
  all $s\in[0,L]$, and for each $\gamma\in\Bbb{G}$ we know the first-,
  second-, and third-order partial derivatives of $g_{\gamma}(s,r)$.
  Then the set of reconstructible $\gamma\in\Bbb{G}$ is
  $C^1$ dense in $\Bbb{G}$ where reconstructability is modulo reparametrization, translation, and rotation.
\end{thm}

\begin{figure}[htp!]
  \centering
  \input{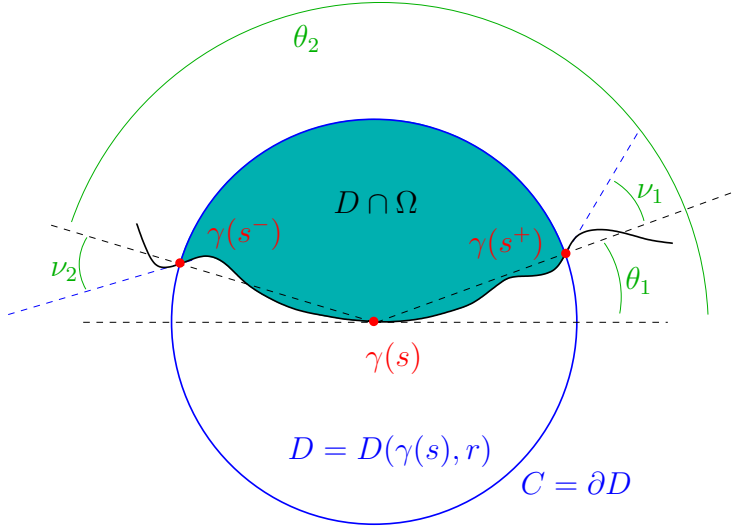}
  \caption{Figure~\ref{fig:step0} again as a reminder}
\label{fig:ref-fig}
\end{figure}

\prff In section \ref{sec:derivatives} we showed that $\frac{\partial
  g(s,r)}{\partial r} = r(\theta_2 - \theta_1)$ and $\frac{\partial
  g(s,r)}{\partial s} = r(\sin(\theta_2) - \sin(\theta_1))$, where the notation is as in Figure \ref{fig:ref-fig}.
Because $\gamma$ is TGL, we can solve for $\theta_1$ and $\theta_2$ from these two derivatives as in the proof of Theorem \ref{thm:tcgl-poly-reconstruct}.

\begin{clm}
\label{clm:2ndderiv}
The following equations hold:
  $\frac{\partial^2 g(s,r)}{\partial r^2} = \theta_2 - \theta_1 +
  r(\frac{\partial\theta_2}{\partial r} -
  \frac{\partial\theta_1}{\partial r})$ and $\frac{\partial^2
    g(s,r)}{\partial r\partial s} = \sin(\theta_2) - \sin(\theta_1) +
  r(\cos(\theta_2)\frac{\partial\theta_2}{\partial r} -
  \cos(\theta_1)\frac{\partial\theta_1}{\partial r} )$.
\end{clm}

\prfclm{clm:2ndderiv} Simply differentiate the expressions we already have for
$\frac{\partial g(s,r)}{\partial r}$ and $\frac{\partial
  g(s,r)}{\partial s}$. \epfclm 

We wish to express this in terms of $\nu_1$ and $\nu_2$.
Note that if we expand the circle radius by $\Delta r$, the right exit point $s_+(s)$ moves approximately (i.e., considering first-order terms only) a distance of $k \equiv \Delta r \sec(\nu_1)$ (so $\pd{k}{r} = \sec{\nu_1}$, a fact we will use later to compute curvature).
Therefore,
\[
\pd{\theta_1}{r} = \lim_{\Delta r \rightarrow 0} \frac{\arctan\left( \frac{r\sin\theta_1+k\sin(\theta_1+\nu_1)}{r\cos\theta_1 + k\cos(\theta_1 + \nu_1)}\right)-\theta_1}{\Delta r}.
\]
Straightforward techniques yield 
$
\pd{\theta_1}{r} %
= \frac{\tan{\nu_1}}{r}
$ and a similar calculation shows that $\pd{\theta_2}{r} = \frac{\tan\nu_2}{r}$.

Therefore, rewriting the second derivatives of $g(s,r)$ in terms of $\nu_1$ and $\nu_2$, we get:
\begin{eqnarray*}
  \frac{\partial^2 g(s,r)}{\partial r^2} & = & \theta_2 - \theta_1 +
  \tan(\nu_2) - \tan(\nu_1) \\
  \frac{\partial^2
    g(s,r)}{\partial r\partial s} & = & \sin(\theta_2) - \sin(\theta_1) +
  \cos(\theta_2)\tan(\nu_2) -
  \cos(\theta_1)\tan(\nu_1)
\end{eqnarray*}
Using these 2 derivatives, together with the previous two, we can
solve for $\nu_1 = \arctan(r \frac{\partial\theta_1}{\partial r})$ and
$\nu_2 = \arctan(r \frac{\partial\theta_1}{\partial r})$ whenever
$\cos(\theta_1) \neq \cos(\theta_2)$. Since we are assuming that the
curve is a simple closed curve, $\cos(\theta_1) \neq \cos(\theta_2)$ is
always true.

\begin{clm}
\label{clm:3rdderiv}
  Knowing $\frac{\partial^3 g(s,r)}{\partial r^3}$ and
  $\frac{\partial^3 g(s,r)}{\partial r^2\partial s}$ gives us
  $\kappa(s^+(s))$ and $\kappa(s^-(s))$, the curvatures of $\gamma$ at
  $s^+(s)$ and $s^-(s)$.
\end{clm}
\prfclm{clm:3rdderiv} Computing, we get
\begin{eqnarray*}
  \frac{\partial^3 g(s,r)}{\partial r^3} & = & \frac{\partial \theta_2}{\partial r} - \frac{\partial \theta_1}{\partial r} +
  \sec^2(\nu_2)\frac{\partial \nu_2}{\partial r} - \sec^2(\nu_1)\frac{\partial \nu_1}{\partial r} \\
  \frac{\partial^3
    g(s,r)}{\partial r^2\partial s} & = & \cos(\theta_2)\frac{\partial \theta_2}{\partial r} - \cos(\theta_1)\frac{\partial \theta_1}{\partial r} 
  - \sin(\theta_2)\frac{\partial \theta_2}{\partial r}\tan(\nu_2) \\ & + &
  \sin(\theta_1)\frac{\partial \theta_1}{\partial r}\tan(\nu_1) + \cos(\theta_2)\sec^2(\nu_2)\frac{\partial \nu_2}{\partial r} -
  \cos(\theta_1)\sec^2(\nu_1)\frac{\partial \nu_1}{\partial r}.
\end{eqnarray*} 
Since $\nu_2'\equiv\frac{\partial \nu_2}{\partial r}$ and
$\nu_1'\equiv\frac{\partial \nu_1}{\partial r}$ are the only unknowns,
we end up having to invert
\begin{equation*}
  \left[
    \begin{array}{cc}
      1 & -1 \\
     \cos(\theta_2) & \cos(\theta_1)
    \end{array}
\right]
\end{equation*} 
again and this is always nonsingular, giving us $\nu_1'$ and $\nu_2'$
as a function of s, the coordinate of the center of the disk.

Relative to the horizontal, the angle of the curve at $s^+(s)$ is $\theta_1 + \nu_1$ so the rate of change in angle as we expand the circle is $\pd{\theta_1}{r}+\nu_1'$.
Recalling that rate of movement of this exit point as we expand the circle is given by $\pd{k}{r} = \sec \nu_1$, we have that the curvature is given by $\kappa(s^+(s))= \pd{k}{r}(\pd{\theta_1}{r}+\nu_1') = \sec \nu_1(\pd{\theta_1}{r}+\nu_1')$.
Similarly, $\kappa(s^-(s)) = \sec(\nu_2)(\pd{\theta_2}{r}+\nu_2')$.
\epfclm
\begin{clm}
  Generically, we can deduce $s^+(s)$ from knowledge of $\nu_1(s)$,
  $\nu_2(s)$, $\theta_1(s)$ and $\theta_2(s)$.
\end{clm}

\prff We outline the proof without some of the explicit
constructions that follow without much trouble from the outline. We
have that $\theta_1(s^-(s)) + \nu_1(s^-(s))= \pi - \theta_2(s) -
\nu_2(s)$ and $\theta_1(s) + \nu_1(s) = \pi - \theta_2(s^+(s)) -
\nu_2(s^+(s))$. All four of these quantities (the left- and right-hand sides of each of
the 2 equations) are the turning angles between the tangent to the
curve at the center of the disk and the tangent to the curve at a
point $r$ away from the center of the disk.

Now we use this correspondence between the $\theta+\nu$ curves to
solve for $s^-(s)$ and $s^+(s)$. But these curves can differ by a
homeomorphism of the domain. Thus, we can only find the correspondence
if there is a distinguished point on those curves as well as no places
where the values attained are constant. The turning angle curves
having isolated critical points and a unique maximum or minimum is
sufficient for our purposes.

To get isolated extrema, start by approximating the curve $\gamma$
with another one, $\hat{\gamma}$, that agrees in $C^1$ at a large but
finite number of points $\{s_i\}_{i=1}^N$ (i.e. agrees in tangent
direction as well as position) and has isolated critical points in
the derivative of the tangent direction. Now perturb $\hat{\gamma}$ to
one that is $C^1$ close (but not $C^2$ close) by using oscillations
about the curve so that the 2nd and 3rd derivatives are never
simultaneously below the bounds on the 2nd and 3rd derivatives of the
curve we started with. We do this in a way that alternates around the
curve. See Figure~\ref{fig:corre_perturb}.  In a bit more detail,
suppose that $\max\{d^2\hat{\gamma}/ds^2, d^3\hat{\gamma}/ds^3\} <
L_1$. Choose a starting point on the curve; $s=0$ works. Now begin
perturbing $\hat{\gamma}$ at the point $s_{\hat{r}}$ in the positive
$s$ direction such that $|\hat{\gamma}(s_{\hat{r}}) - \hat{\gamma}(0)|
= \hat{r}$. We name the newly perturbed curve $\hat{\hat{\gamma}}$ and
we keep $L_1 < \max\{ d^2\hat{\hat{\gamma}}/ds^2,
d^3\hat{\hat{\gamma}}/ds^3 \} < L_2$.  We continue perturbing until we
have reached $s_{2r}$ defined by $|\hat{\gamma}(s_{2\hat{r}}) -
\hat{\gamma}(s_{\hat{r}})| = \hat{r}$. We begin perturbing again when
we reach $s_{3\hat{r}}$. Continue in this fashion around
$\hat{\gamma}$. The last piece, shown in green in the figure, will
require a perturbation that is distinct in size due to the fact that
it will interact with the perturbation that starts at $s_{\hat{r}}$.
On this last piece, we enforce $L_2 < \max\{
d^2\hat{\hat{\gamma}}/ds^2, d^3\hat{\hat{\gamma}}/ds^3 \} < L_3$. All
these perturbations can be chosen with isolated singularities in
derivatives, thus giving us $\theta +\nu$ curves that are monotonic
between isolated singularities.  (In fact, we might as well choose all
perturbations to be piecewise polynomial perturbations. This
immediately gives us the isolated singularities and monotonicity that
we want.)

Finally, if there is not a distinct maximum, we can choose one of the
maxima and add a small twist to the curve at that point. See
Figure~\ref{fig:twist_perturb}. The idea is that a small twist,
applied to the leading edge of the tangents we are comparing to get the
turning angle, will increase the angle most at the center of the
twist. If this corresponds to a nonunique global maximum, we end up
with a unique global maximum.

\begin{figure}[htp!]
  \centering
  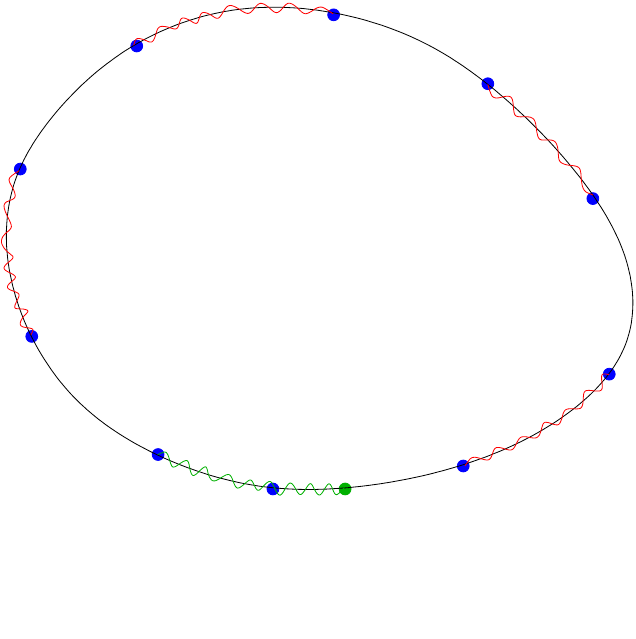
  \caption{In this schematic figure, we illustrate the alternating
    perturbation around the curve, keeping the curve $C^1$ close to
    and messing with the second and third derivatives to eliminate any
    critical points other than isolated maxima and minima. Here the
    perturbation is of course greatly exaggerated. }
 \label{fig:corre_perturb}
\end{figure}

\begin{figure}[htp!]
  \centering
  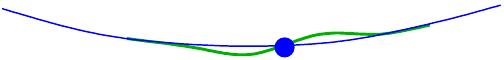
  \caption{A twist perturbation. Notice that if the twist is applied
    precisely at a global max of the turning angle (as measured by the
    tangent here and the one lagging it in $s$), we will increase the
    turning angle there and will end up with a unique global maximum.
  }
 \label{fig:twist_perturb}
\end{figure}

Now the correspondence scheme works. That is, we know that the global
maximums must match, and because the turning angle curves are
monotonic between isolated critical points, we can find the
homeomorphisms in $s$ that move the turning angle curves into
correspondence.  
\epfclm

Taken together, the last two claims give us the curvature as a
function of arclength. This determines $\gamma$ up to translations and
rotations.  \epf

\section{Numerical experiments}
\label{sec:numerics}

In this section, we consider a numerical curve reconstruction for the situation in which $g(s,r)$ is known for a given radius $r$ but no derivative information is available.  This reconstruction is more strict than the scenarios of sections \ref{sec:reconstruct-T}--\ref{sec:generic}.  Our motivation is to explore whether any $\gamma$ can be uniquely and practically reconstructed with this limited information.  

We consider $\gamma_a(\bar{s})\in \mathcal{P}^N$, the set of \textit{simple} polygons of $N$ ordered vertices $\{(x_1,y_1), \dots, (x_{N},y_{N})\}$ parameterized by the set $\{\bar{s}_k\}_{k=1}^{N}$ with $\bar{s}_k=k/N$ as 
\begin{equation}
\begin{array}{l}
 \displaystyle x_k = \sum_{j=0}^{m-1}{a_{1,j} \cos(2\pi j \bar{s}_k/N)+a_{2,j} \sin(2\pi j \bar{s}_k/N)},\\
 \displaystyle y_k = \sum_{j=0}^{m-1}{a_{3,j} \cos(2\pi j \bar{s}_k/N)+a_{4,j} \sin(2\pi j \bar{s}_k/N)},
\end{array}
\end{equation}
for some coefficients $a_{i,j}\in\mathbb{R}$.  In this way, the polygon $\gamma$ is a discrete approximation of a $C^{\infty}$ curve.  The sides of $\gamma_a(\bar{s})$ are not necessarily of equal length.

We take the vector signature $g_a(\bar{s},r)\in\mathbb{R}^N$ to be the discrete area densities of $\gamma_a(\bar{s})$ computed at each vertex.  Given such a signature for fixed radius $r$ and fixed partition $\bar{s}$, we seek $a^*$ satisfying
\begin{equation}
\label{equ:numopt}
\begin{array}{rl}
 \displaystyle  a^* \in &  \displaystyle \arg\min_{b\in\mathbb{R}^{4m}}\|g_b(\bar{s},r)-g_a(\bar{s},r)\|_2^2 \\
  & \\
  &  \displaystyle \text{s.t.  } \gamma_b\in\mathcal{P}^N
\end{array}
\end{equation}
Equation (\ref{equ:numopt}) represents a nonlinearly constrained optimization problem with continuous nonsmooth objective.  The constraint ensures that polygons are simple though any optimal reconstruction $\gamma_{a^*}$ is not expected to lie on the feasible region boundary except in cases of noisy signatures.  This approach to reconstructing curves seeks a polygon that matches a given discrete signature, rather than an analytic sequential point construction procedure.

We use the direct search \textsc{OrthoMads} algorithm\cite{AbAuDeLe2009} to solve this problem.  \textsc{Mads} class algorithms do not require objective derivative information\cite{AbAuDeLe2009,AuDe2006} and converge to second-order stationary points under reasonable conditions on nonsmooth functions\cite{AbAu2006}.  We implement our constraint using the extreme barrier method\cite{AuDeLe2009} in which the objective value is set to infinity whenever constraints are not satisfied. 
We utilize the standard implementation with partial polling and minimal spanning sets of $4m+1$ directions.

\begin{figure}[h!]
\centering
\includegraphics[width=.6\textwidth]{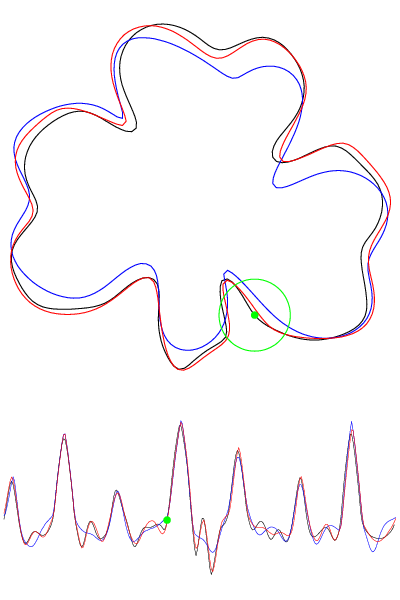}
\caption{Shamrock reconstruction: comparing the original curve with those found for {\color{blue}$m = 12$} and {\color{red}$m=18$}.  Curves for $m \geq 20$ are visually indistinguishable from the original curve.  The shape signatures are given at the bottom.}
\label{fig:shamrock}
\end{figure}

We performed a series of numerical tests using the synthetic shamrock curve shown in black in the upper portion of Figure~\ref{fig:shamrock}.  This curve is given as a polygon in $\mathcal{P}^{256}$ with discretization coefficients $a\in\mathbb{R}^{4\times 20}$ ($m=20$).  A sequence of reconstructions was performed with all integer values $8\leq m \leq 20$.  The $m=8$ reconstruction begins with initial coefficients, $a_{i,j}$, which determine a regular $256$-gon with approximately the same interior area as the shamrock (as determined by the signature $g_a(\bar{s},r)$.
In particular, the value(s) $a_{i,j}$ supplied initially are those which define the best fit circle ($m=1$),
which can be computed directly.
That is, only $a_{1,0}$ and $a_{4,0}$ are nonzero.
Subsequent reconstructions begin with initial coefficients optimal to the previous relatively coarse reconstruction.  Curve reconstructions for $m=12$ (blue) and $m=18$ (red) are compared to the shamrock in the upper portion of Figure~\ref{fig:shamrock}.   Reconstructions for $m\geq20$ are visually indistinguishable from the actual curve and are not shown.  Corresponding area density signatures are shown in the lower portion of Figure~\ref{fig:shamrock}.  A representative disk of radius $r$ is shown in green along with corresponding location in the signature; note that the shamrock is not \tcgl{} with this radius.

When comparing and interpreting the shamrock curves, it is important to note that the scale of the curves is determined entirely by the fit parameters $a_{i,j}$.  On the other hand, as the density signature is independent of curve rotation, the rotation is eyeball adjusted for easy visual comparison.  
Also note that the two-arc property does not hold for this example so our reconstructability results do not apply.
The accuracies of both the curve reconstruction and area density signature fit suggest that somewhat more general reconstructability results hold.
In particular, we speculate that general simple polygons may be reconstructible from $g(s,r)$ for fixed $r$ and no derivative information.


\section{Conclusions}
We have studied the integral area invariant with particular emphasis on the \tcgl{} condition.
In particular, we have shown that all TCGL polygons and a $C^1$-dense set of $C^2$ TGL curves are reconstructible using only the integral area invariant for a fixed radius along the boundary and its derivatives.

We also showed that TCGL boundaries can be approximated by TCGL polygons, determined what the derivatives represented, and commented on other sets of data sufficient for reconstruction (namely, both T-like and all radii in a neighborhood of 0).

These reconstructions are all modulo translations, rotations, and reparametrizations.
The arc length parameterization plays a special role here since any two such parameterizations of a boundary will differ only by a shift and can easily be placed into correspondence.
The situation becomes more complicated in higher dimensions as boundaries are no longer canonically parameterized by a single variable which is a fundamental assumption of our results and methods.
It is not immediately obvious how to resolve the issues created by higher dimensions except that it may be possible to modify some of the machinery to work with star convex regions which restore some semblance of canonical representation.

Another space which is open for further development is that of reconstruction algorithms.
This is doubly true since our theoretical reconstructions are unstable and the numerical examples in the present work do not have guaranteed reconstruction.
However, even without these guarantees, the numerical examples hint at more expansive reconstructability results.

\section{Acknowledgments}
The authors would like to thank David Caraballo for introducing us to this topic as well as Simon Morgan and William Meyerson for initial discussions and work on related topics that are not in this paper.
This research was supported in part by National Science Foundation grant DMS-0914809.

\appendix
\section{Appendix: Easy Reconstructability}
\label{sec:appen}

For completeness, we include a short proof of the fact that knowing
$g(s,r)$ for all $s$ and $r$ very easily gives us reconstructability.
This follows from the fact that knowing the asymptotic behavior of
$g(s,r)$ as $r\rightarrow 0$ for any $s$ gives us $\kappa(s)$. That in
turn implies that knowing $g(s,r)$ in any neighborhood of the set
$(s,r)\in [0,L] \times \{r=0\}$ also gives us $\kappa(s)$ and
therefore the curve.

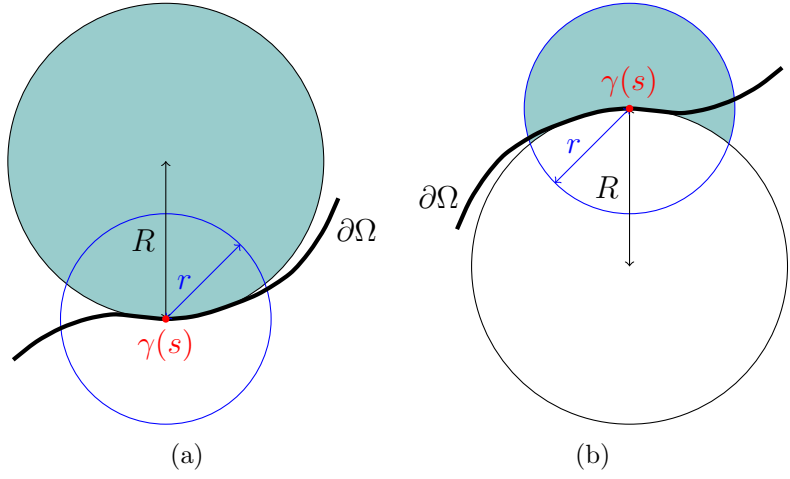
\begin{figure}[htp!]
\centering
\subfigure[]{
\label{fig:pos-osculating-circle}
\begin{tikzpicture}[scale=0.7]
\fill[color=teal!40] (0,3) circle (3);
\draw[color=black] (0,3) circle (3);
\draw[color=blue] (0,0) circle (2);
\draw[ultra thick] plot[smooth] coordinates {(195:3) (190:2.5) (185:2) (180:1.5) (175:1) (175:0.5) (0,0) (5:0.5) (10:1) (15:1.8) (20:2.5) (25:3) (30:3.5) (35:4)};
\node () at (25:4) {$\bd$};
\draw[<->] (0:0) -- (90:3) node[pos=0.5,left] {$R$};
\draw[color=blue,<->] (0:0) -- (45:2) node[pos=0.5,left] {$r$};
\fill[color=red] (0,0) circle (2pt);
\node[color=red] () at (0,-.5) {$\gamma(s)$};
\end{tikzpicture}}
\subfigure[]{
\label{fig:neg-osculating-circle}
\begin{tikzpicture}[scale=0.7,rotate=180]
\fill[color=teal!40] (0,0) circle (2);
\fill[color=white] (0,3) circle (3);
\draw[color=black] (0,3) circle (3);
\draw[color=blue] (0,0) circle (2);
\draw[ultra thick] plot[smooth] coordinates {(195:3) (190:2.5) (185:2) (180:1.5) (175:1) (175:0.5) (0,0) (5:0.5) (10:1) (15:1.8) (20:2.5) (25:3) (30:3.5) (35:4)};
\node () at (25:4) {$\bd$};
\draw[<->] (0:0) -- (90:3) node[pos=0.5,left] {$R$};
\draw[color=blue,<->] (0:0) -- (45:2) node[pos=0.5,left] {$r$};
\fill[color=red] (0,0) circle (2pt);
\node[color=red] () at (0,-.5) {$\gamma(s)$};
\end{tikzpicture}
}
\caption{Using the osculating circle as a surrogate for $\bd$ in the \subref{fig:pos-osculating-circle} positive and \subref{fig:neg-osculating-circle} negative curvature cases.}
\label{fig:osculating-circle}
\end{figure}
\begin{thm}
Suppose $\bd$ is $C^2$ and there exists $\epsilon > 0$ such that we know $g(s,r)$ for all $(s,r) \in [0,L) \times (0,\epsilon)$.
This information is enough to determine the curvature of every point on $\bd$.
In particular, if $\gamma : [0,L) \rightarrow \bd$ is a counterclockwise arclength parameterization of $\bd$, then $\kappa(\gamma(s)) = -3\pi \lim_{r\rightarrow 0} \pd{}{r} \frac{g(s,r)}{\pi r^2}$.
\begin{proof}
Fix $s \in [0,L)$.
If the curvature of $\gamma$ at $s$ is positive, we consider what happens if we replace $\shape$ with the disk whose boundary is the osculating circle of $\bd$ at $\gamma(s)$ (call its radius $R$).
We have the following expression for the new normalized nonasymptotic density (see Figure \ref{fig:pos-osculating-circle}):
\[
\frac{g(s,r)}{\pi r^2} = \frac{1}{\pi r^2}\int_{-p}^{p} \sqrt{r^2-x^2}-(R-\sqrt{R^2-x^2})\, dx.
\]
where $x = p$ is the positive solution to $\sqrt{r^2-x^2} = R-\sqrt{R^2-x^2}$.
Differentiating with respect to $r$ and then taking the limit as $r$ goes to 0 gives us $-\frac{1}{3\pi R}$.
That is, for the case where $\shape$ is locally a disk, the curvature at $\gamma(s)$ is given by $-3\pi\lim_{r\rightarrow 0}\pd{}{r}\frac{g(s,r)}{\pi r^2}$.

If the curvature of $\bd$ at $\gamma(s)$ is negative, we can set up a similar surrogate (see figure \ref{fig:neg-osculating-circle}) and again obtain that $\kappa(\gamma(s)) = -3\pi \lim_{r\rightarrow 0} \pd{}{r} \frac{g(s,r)}{\pi r^2}$.

Lastly, this calculation gives the right result in the curvature 0 case when $\bd$ is locally a straight line (so $\frac{g(s,r)}{\pi r^2} = \frac{1}{\pi r^2}\int_{-r}^{r}\sqrt{r^2-x^2}\,dx = \frac{1}{2}$ for sufficiently small $r$ and $-3\pi\lim_{r\rightarrow 0} \pd{}{r} \frac{g(s,r)}{\pi r^2} = 0$).

For the case where $\bd$ is not locally a circle or straight line, the corrections to the integrals are of order $O(x^3)$ as $r$ goes to 0 and have no impact on the final answer so the curvature at $\gamma(s)$ is always given by $-3\pi\lim_{r\rightarrow 0}\pd{}{r}\frac{g(s,r)}{\pi r^2}$.
The available data (the values $g(s,r)$ for all $s \in [0,L)$ and all $r \in (0, \epsilon)$) are sufficient to compute the relevant derivative and limit so we can use this process to determine the curvature of every point on the $C^2$ curve $\bd$.
\end{proof}
\end{thm}

\bibliographystyle{plain}
\bibliography{geo}

\end{document}